\documentclass[11pt]{amsart}

\usepackage{amssymb,latexsym,amsmath,amsfonts,amsthm, eucal}
\input{amssym.def}
\date{}

\def\C{{\mathbb C}}

\def\N{\mathbb{N}}

\def\ra{\rightarrow}
\def\ov{\overline}
\def\lo{\longrightarrow}

\def\a{\alpha}

\def\g{\gamma}
\def\d{\sum}

\def\beq{\begin{eqnarray}}
\def\eeq{\end{eqnarray}}
\def\beqa{\begin{eqnarray*}}
\def\eeqa{\end{eqnarray*}}
\def\wt{\widetilde}

 \newtheorem{thm}{Theorem}[section]
 
 \newtheorem{lem}[thm]{Lemma}
 \newtheorem{prop}[thm]{Proposition}
 \theoremstyle{definition}
 
 \newtheorem{rem}[thm]{Remark}
 \newtheorem{ex}[thm]{Example}
 \numberwithin{equation}{section}

\begin{document}
\title[Hilbert modules]{Resolution of singularities for a class of Hilbert modules}

\author[Biswas]{Shibananda Biswas}

\author[Misra]{Gadadhar Misra}
\address[Shibananda Biswas and Gadadhar Misra]{Department of Mathematics, Indian Institute of Science, Banaglore 560012}

\email[S. Biswas]{shibu@math.iisc.ernet.in} \email[G. Misra]{gm@math.iisc.ernet.in}

\thanks{Financial support for the work of S. Biswas was provided in the
form of a Research Fellowship of the Indian Statistical Institute and the
Department of Science and Technology.
The work of G. Misra was  supported in part by UGC - SAP and
by a grant from the Department of Science and Technology.}



\begin{abstract}
A short proof of the ``Rigidity theorem'' using the sheaf theoretic model for Hilbert modules over polynomial rings is given. The joint kernel for a large class of submodules is described. The completion $[\mathcal I]$ of a homogeneous (polynomial) ideal $\mathcal I$ in a Hilbert module is a submodule for which the joint kernel is shown to be of the form
$$
\{p_i(\tfrac{\partial}{\partial \bar{w}_1}, \ldots , \tfrac{\partial}{\partial \bar{w}_m}) K_{[\mathcal I]}(\cdot,w)|_{w=0}, 1 \leq i \leq n\},
$$
where $K_{[\mathcal I]}$ is the reproducing kernel for the submodule $[\mathcal I]$ and $p_1, \ldots , p_n$ is some minimal ``canonical set of  generators'' for the ideal $\mathcal I$.
The proof includes an algorithm for constructing this canonical set of generators,  which is  determined uniquely modulo linear relations, for homogeneous ideals.  A set of easily computable invariants for these
submodules, using the monoidal transformation, are provided.
Several examples are given to illustrate the explicit computation of these invariants.
\end{abstract}

\maketitle

\section{Preliminaries}\label{intro}
Beurling's theorem describing the invariant subspaces of the multiplication (by the coordinate function) operator on the Hardy space of the unit disc is essential to the Sz.-Nagy -- Foias model theory and several other developments in modern operator theory.  In the language of Hilbert modules,  Beurling's theorem says that all submodules of the Hardy module of the unit disc are equivalent.  This observation, due to Cowen and Douglas \cite{cd2}, is peculiar to the case of one-variable operator theory. The submodule of functions vanishing at the origin of the Hardy module $H_0^2(\mathbb D^2)$ of the bi-disc is not equivalent to the Hardy module $H^2(\mathbb D^2)$.
To see this, it is enough to note that the joint kernel of the adjoint of the multiplication by the two co-ordinate functions on the Hardy module of the bi-disc  is $1$ - dimensional (it is spanned by the constant function $1$) while the joint kernel of these operators restricted to the submodule is $2$ - dimensional (it is spanned by the two functions $z_1$ and $z_2$).

There has been a systematic study of this phenomenon in the recent past \cite{as,dpsy} resulting in a number of ``Rigidity theorems'' for submodules of a Hilbert module $\mathcal M$ over the polynomial ring $\C[\underline z]:=\C[z_1,\ldots,z_m]$ of the form $[\mathcal I]$ obtained by taking the norm closure of a polynomial ideal $\mathcal I$ in the Hilbert module. For a large class of polynomial ideals, these theorems often take the form: two submodules $[\mathcal I]$ and $[\mathcal J]$ in some Hilbert module $\mathcal M$ are equivalent if and only if the two ideals $\mathcal I$ and $\mathcal J$ are equal. We give a short proof of this
theorem  using the sheaf theoretic model developed earlier in \cite{bmp} and construct tractable invariants for Hilbert modules over $\C[\underline z]$.




Let $\mathcal M$ be a Hilbert module of holomorphic functions on a bounded open connected subset $\Omega$ of $\mathbb C^m$ possessing a reproducing kernel $K$.  Assume that $\mathcal I \subseteq \mathbb C[\underline{z}]$ is the singly generated ideal $\langle p \rangle$.
Then the reproducing kernel $K_{[\mathcal I]}$ of $[\mathcal I]$  vanishes on the zero set $V(\mathcal I)$
and the map $w \mapsto K_{[\mathcal I]}(\cdot , w)$  defines a holomorphic Hermitian line bundle on the open set $\Omega_\mathcal I^*=\{w\in \mathbb C^m: \bar{w} \in \Omega\setminus V(\mathcal I)\}$
which naturally extends to all of $\Omega^*$.
As is well known, the curvature of this line bundle completely determines the equivalence class of the Hilbert module $[\mathcal I]$ (cf. \cite{cd, cd1}). However, if $\mathcal I\subseteq \mathbb C[\underline{z}]$ is not a principal ideal, then the corresponding line bundle defined on $\Omega_\mathcal I^*$  no longer extends to all of $\Omega^*$.    Indeed, it was conjectured in \cite{dmv} that  the  dimension of the joint kernel of the Hilbert module $[\mathcal I]$ at $w$ is $1$ for points $w$ not in $V(\mathcal I)$,
otherwise it is the codimension of $V(\mathcal I)$. Assuming that
\begin{enumerate}
\item[(a)] $\mathcal I$ is a principal ideal or
\item[(b)] $w$ is a smooth point of $V(\mathcal I)$.
\end{enumerate}
Duan and Guo verify the validity of this conjecture in \cite{dg}.
Furthermore if $m=2$ and $\mathcal I$ is prime then
the conjecture is valid.

Thus for any submodule $[\mathcal I]$ in a Hilbert module
$\mathcal M$, assuming that $\mathcal M$ is in the Cowen-Douglas class $\mathrm B_1(\Omega^*)$ and the co-dimension of  $V(\mathcal I) $ is greater than $1$, it follows that $[\mathcal I]$ is in $\mathrm B_1(\Omega_\mathcal I^*)$ but it doesn't belong to $\mathrm B_1(\Omega^*)$.  For example, $H_0^2(\mathbb D^2)$ is in the Cowen-Douglas class
$\mathrm B_1(\mathbb D^2\setminus \{(0,0)\})$ but it does not belong to $\mathrm B_1(\mathbb D^2)$.  To systematically study examples of submodules like $H^2_0(\mathbb D^2)$, the following definition from \cite{bmp} will be useful.

\noindent {\sf Definition}.
A Hilbert module $\mathcal M$ over the
polynomial ring in $\C[\underline z]$  is said to be in the class
$\mathfrak{B}_1(\Omega^*)$ if
\begin{enumerate}
\item[\sf{(rk)}] possess a reproducing kernel $K$
(we don't rule out the possibility: $K(w,w)=0$ for  $w$ in
some closed subset $X$ of $\Omega$) and \item[\sf{(fin)}] The
dimension of $\mathcal M/\mathfrak m_w\mathcal M$ is finite for all
$w\in \Omega$.
\end{enumerate}

For Hilbert modules in $\mathfrak{B}_1(\Omega)$, from \cite{bmp}, we have:

\noindent{\sf Lemma}.
Suppose $\mathcal M \in \mathfrak B_1(\Omega^*)$ is the closure of a
polynomial ideal $\mathcal I$.  Then $\mathcal M$ is in
$\mathrm{B}_1(\Omega^*)$ if the ideal  $\mathcal I$ is singly
generated while if it is minimally generated by more than one polynomial, then $\mathcal M$ is in $\mathrm{B}_1(\Omega_{\mathcal I}^*)$.

This Lemma ensures that to a Hilbert module in
$\mathfrak B_1(\Omega^*)$, there corresponds a holomorphic Hermitian line
bundle defined by the joint kernel for points in $\Omega_\mathcal I^*$. We will show that it extends to a holomorphic Hermitian line bundle on the ``blow-up'' space $\hat{\Omega}^*$ via the monoidal transform under mild hypothesis on the zero set $V(\mathcal I)$.  We also show that this line bundle determines the equivalence class of the module $[\mathcal I]$ and therefore its curvature is a complete invariant. However, computing it explicitly on all of $\hat{\Omega}^*$ is difficult.
In this paper we find invariants, not necessarily complete, which are  easy to compute. One of these invariants is nothing but the curvature of the restriction of the line bundle on $\hat\Omega^*$ to the exceptional subset of $\hat{\Omega}^*$.

A line bundle is completely determined by its sections on open subsets. To write down the sections, we use the decomposition theorem for the reproducing kernel \cite[Theorem 1.5]{bmp}. The actual computation of the curvature invariant require the explicit calculation of norm of these sections.  Thus it is essential to obtain explicit description of the eigenvectors  $K^{(i)},\,1\leq i\leq d,$ in terms of the reproducing kernel. We give two examples which, we hope, will motivate the results that follow.
Let $H^2(\mathbb D^2)$ be the Hardy module over the bi-disc algebra.
The reproducing kernel for $H^2(\mathbb D^2)$ is the S\"{z}ego kernel $\mathbb S(z,w) = \frac{1}{1-z_1\bar{w}_2}\frac{1}{1-z_2\bar{w}_2}$.
Let $\mathcal I_{0}$ be the polynomial ideal
$\langle z_1, z_2\rangle$ and let $[\mathcal I_{0}]$ denote the minimal closed
submodule of the Hardy module $H^2(\mathbb D^2)$ containing $\mathcal I_{0}$.
Then the joint kernel of
the adjoint of the multiplication operators $M_1$ and $M_2$ is spanned by the two linearly independent vectors: $z_1=p_1(\bar{\partial}_1,\bar{\partial}_2)  \mathbb S(z,w)_{|w_1=0=w_2}$ and  $z_2=p_2(\bar{\partial}_1,\bar{\partial}_2 ) \mathbb S(z,w)_{|w_1=0=w_2}$, where $p_1,p_2$ are the generators of the ideal $\mathcal I_{0}$. For a second example, take the ideal $\mathcal I_{1}=\langle z_1-z_2, z_2^2\rangle$ and let  $[\mathcal I_{1}]$ be the minimal closed
submodule of the Hardy module $H^2_0(\mathbb D^2)$ containing $\mathcal I_{1}$.  The joint kernel is not hard to compute. A set of two linearly independent vectors which span it are
$p_1(\bar{\partial}_1,\bar{\partial}_2)  \mathbb S(z,w)_{|w_1=0=w_2}$ and $p_2(\bar{\partial}_1,\bar{\partial}_2 )  \mathbb S(z,w)_{|w_1=0=w_2}$, where $p_1 = z_1 - z_2$ and $p_2 = (z_1 + z_2)^2$.
Unlike the first example, the two polynomials $p_1, p_2$ are not the generators for the ideal $\mathcal I_1$ that were given at the start, never the less, they are easily seen to be a set of generators for the ideal $\mathcal I_{1}$ as well.  This prompts the question:

\noindent\textsf{Question}: Let $\mathcal M\in\mathfrak B_1(\Omega^*)$ be a Hilbert module and $\mathcal I\subseteq\mathcal M$ be a polynomial ideal. Assume without loss of generality that $0\in V(\mathcal I)$. We ask
\begin{enumerate}
\item \label{q1}if there exists a set of polynomials $p_1, \ldots ,p_n$ such that 
$$p_i(\tfrac{{\partial}}{\partial \bar{w}_1}, \ldots , \tfrac{{\partial}}{\partial \bar{w}_m} )K_{[\mathcal I]}(z,w)_{|z=0=w},\,i=1,\ldots ,n,
$$
spans the joint kernel of $[\mathcal I]$;
\item what conditions, if any, will ensure that the polynomials  $p_1, \ldots , p_n$, as above, is a generating set for $\mathcal I$?
\end{enumerate}
We show that the answer to the Question $(1)$ is affirmative, that is, there is a natural basis for the joint eigenspace of the Hilbert module $[\mathcal I]$,  which is  obtained by applying a differential operator to the reproducing kernel $K_{[\mathcal I]}$ of the Hilbert module $[\mathcal I]$.  Often, these differential operators encode an algorithm for producing a set of generators for the ideal $\mathcal I$ with additional properties. It is shown that there is an  affirmative answer to the Question $(2)$ as well, if the ideal is assumed to be homogeneous. It then follows that, if there were two sets of generators which serve to describe the joint kernel, as above, then these generators must be linear combinations of each other, that is, the sets of generators  are determined modulo a linear transformation.  We will call them \emph{canonical set of generators}. The canonical generators provide an effective tool to determine if two ideal are equal.  A number of examples illustrating this phenomenon is given.

In the following section, we describe the joint kernel. In section \ref{res}, we construct the holomorphic Hermitian line bundle on the ``blow - up '' space. In the last section, we provide an explicit calculation.

\subsection{\rm }{\sf Index of notations:}\vskip .75em

\begin{tabular}{ll}
$\C[\underline z]$ & the polynomial ring $\C[z_1,\ldots,z_m]$ of $m$- complex variables  \\

$\mathfrak m_w$ & maximal ideal of $\C[\underline z]$ at the point $w\in\C^m$\\

$\Omega$ & a bounded domain in $\C^m$ \\

$\Omega^*$ & $\{\bar z: z\in\Omega\}$ \\

$\mathbb D$ & the open unit disc in $\C$\\

$\mathbb D^m$ & the poly-disc $\{z\in\C^m:|z_i|<1,\, 1\leq i\leq t\},m\geq 1$\\

$[\mathcal I]$ & the completion of a polynomial ideal $\mathcal I$ in some Hilbert module\\

$M_i$ & module multiplication by the co-ordinate function $z_i$ on $[\mathcal I]$, $1\leq i\leq m$\\

$M_i^*$ & adjoint of the multiplication operator $M_i$ on $[\mathcal I]$, $1\leq i\leq m$\\

$K_{[\mathcal I]}$ & the reproducing kernel of $[\mathcal I]$\\

$\alpha, |\alpha|, \alpha !$ & the multi index $(\alpha_1,\ldots,\alpha_m)$, $|\alpha|=\sum_{i=1}^m\alpha_i$ and $\alpha ! = \alpha_1!\ldots\alpha_m!$\\

$\binom{\alpha}{k}$ &  $= \prod_{i=1}^m\binom{\alpha_i}{k_i}$ for $\alpha = (\alpha_1,\ldots,\alpha_m)$ and  $k=(k_1,\ldots,k_m)$\\

$k\leq \alpha$ & if $k_i\leq\alpha_i$,  $1\leq i\leq m$.\\

$z^\alpha$ & $z_1^{\alpha_1}\ldots z_m^{\alpha_m}$\\

$\partial^{\alpha},\bar\partial^{\alpha}$ & $\partial^{\alpha}=\frac{\partial^{|\alpha|}}{\partial z_1^{\alpha_1}\cdots z_m^{\alpha_m}},\bar{\partial}^{\alpha}=\frac{\partial^{|\alpha|}}{\partial\bar{z}_1^{\alpha_1}\cdots \bar{z}_m^{\alpha_m}}$ for $\alpha \in
{\mathbb Z^+\times\cdots\times\mathbb Z^+}$\\
$q(D)$ & the differential operator $q(\tfrac{\partial}{\partial z_1}, \ldots , \tfrac{\partial}{\partial z_m})$  ( = $\sum_{\alpha}a_{\alpha}\partial^{\alpha}$, where $q = \sum_{\alpha}a_{\alpha}z^{\alpha}$) \\


$ B_n(\Omega)$ & Cowen-Douglas  class of operators of rank $n$, $n\geq 1$\\

$q^*$ & $q^*(z) = \ov{q(\bar z)} (= \sum_{\alpha}\bar a_{\alpha}z^{\alpha}$ for  $q$ of the form   $\sum_{\alpha}a_{\alpha}z^{\alpha})$ \\

$\langle~ ,~ \rangle_{w_0}$ & the Fock inner product at $w_0$, defined by $\langle p, q\rangle_{w_0}:= q^*(D)p|_{w_0}
= (q^*(D)p)(w_0)$\\

$\mathcal S^{\mathcal M}$ & the analytic subsheaf of $\mathcal O_\Omega$, corresponding to the Hilbert module $\mathcal M \in \mathfrak{B}_1(\Omega^*)$ \\

$\mathbb V_w(\mathcal F)$ & the characteristic space at $w$, which is $\{q\in\C[\underline z]: q(D)f\big{|}_w=0 \mbox{~for ~all ~}f\in\mathcal F\}$\\
& for some set  $\mathcal F$ of holomorphic functions\\

\end{tabular}

\section{Calculation of basis vectors for the joint kernel}
The Fock inner product of a pair of polynomials $p$ and $q$ is defined by the rule:
$$\langle p, q\rangle_0=q^*(\tfrac{\partial}{\partial z_1}, \ldots , \tfrac{\partial}{\partial z_m})\,p|_0,\,\, q^*(z) = \overline{q(\bar{z})}.$$
The map $\langle ~,~ \rangle_0:\C[\underline{z}]\times \C[\underline{z}]\lo\C$ is linear in first variable and
conjugate linear in the second and for $p=\d_{\alpha} a_{\alpha}
z^{\alpha},~q=\d_{\alpha} b_{\alpha}z^{\alpha}$ in $\C[\underline{z}]$,  we have
$$
\langle p, q\rangle_0 = \d_{\alpha} \alpha ! a_{\alpha}\bar b_{\alpha}
$$
since $z^{\alpha}(D)z^{\beta}|_{z=0}= \alpha!$ if
$\alpha=\beta$ and $0$ otherwise.  Also,  $\langle p, p\rangle_0 = \d_{\alpha} \alpha ! |a_{\alpha}|^2\geq 0$ and equals $0$ only when 
$a_{\alpha}=0$ for all $\alpha$.  The completion of the polynomial ring with this inner product is the well known Fock space $L^2_a(\C^m, d\mu)$, that is, the space of all $\mu$-square integrable entire functions on $\C^m$, where
$$
d\mu(z)= {\pi}^{-m}e^{-|z|^2}d\nu(z)
$$
is the Gaussian measure on $\C^m$ ($d\nu$ is the usual Lebesgue measure).

The characteristic space (cf. \cite[page 11]{cg}) of an ideal $\mathcal I$ in $\C[\underline{z}]$ at the point $w$ is the vector space
\begin{eqnarray*}
\mathbb V_w(\mathcal I):= \{q\in\C[\underline{z}] : q(D)p|_w=0,\, p\in\mathcal I\}
&=&\{q\in\C[\underline{z}]: \langle p, q^*\rangle_w = 0,\, p \in \mathcal I\}.
\end{eqnarray*}
The envelope of the ideal $\mathcal I$ at the point $w$ is defined to be the ideal
\beqa
\mathcal I^e_w &:=& \{p\in\C[\underline{z}] : q(D)p|_w=0,\,q\in\mathbb V_w(\mathcal I)\}\\ &~=& \{p\in\C[\underline{z}] : \langle p, q^*\rangle_w = 0,\,q \in \mathbb V_w(\mathcal I)\}.
\eeqa
It is known \cite[Theorem 2.1.1, page 13]{cg} that $\mathcal I= \cap_{w\in V(\mathcal I)}\mathcal I_w^e$. The proof makes essential use of the well known Krull's intersection theorem. In particular, if $V(\mathcal I) = \{w\}$, then $\mathcal I_w^e = \mathcal I$.  It is easy to verify this special case using the Fock inner product.   We provide the details below after setting $w=0$, without loss of generality.

Let $\mathfrak m_0$ be the maximal
ideal in $\C[\underline z]$ at $0$.
By Hilbert's Nullstellensatz, there exists
a positive integer $N$ such that $\mathfrak m_0^N\subseteq\mathcal I$. We identify $\C[\underline{z}]/\mathfrak m_0^N$  with
$\mbox{span}_{\C}\{z^\alpha:|\alpha|<N\}$ which is the same as $(\mathfrak{m}_0^N)^{\perp}$ in the Fock inner product. 
Let $\mathcal I_N$ be the vector space $\mathcal I\cap \mbox{span}_\C \{z^\alpha:|\alpha|<N\}$. Clearly $\mathcal I$ is the vector space (orthogonal) direct sum $\mathcal I_N\oplus\mathfrak m_0^N$.  Let
\begin{eqnarray*}
\tilde V = \{q \in \C[\underline{z}]: \mbox{deg}\,\, q < N \,\mbox{~and~} \langle p,q \rangle_0 =0,\,p\in \mathcal I_N\}=\big (\mathfrak m_0^N\big )^{\perp} \ominus \mathcal I_N. \end{eqnarray*}
Evidently, $\mathbb{V}_0(\mathcal I) = \tilde V^*$, where $\tilde V^*=\{q\in V:q^*\in\tilde V\}$. It is therefore clear that the definition of $\tilde V$ is independent of $N$, that is, if $\mathfrak m^{N_1} \subset\mathcal I$ for some $N_1$, then $(\mathfrak m_0^{N_1} )^{\perp} \ominus \mathcal I_{N_1} = (\mathfrak m_0^N )^{\perp} \ominus \mathcal I_N$. Thus
\begin{eqnarray*}
\mathcal I^e_0 &=&
\{p\in\C[\underline{z}] : \mbox{deg}\,\, p < N\mbox{~and~}\langle p, q^*\rangle_0=0,\, q\in\mathbb V_0(\mathcal I)\} \oplus \mathfrak m_0^N\\
&=&\big ( (\mathfrak m_0^N)^\perp \ominus \tilde{V}\big )\oplus\mathfrak m_0^N\\
&=& \mathcal I_N \oplus \mathfrak m_0^N
\end{eqnarray*}
showing that $\mathcal I^e_0 = \mathcal I$.

Let $\mathcal M$ be a submodule of an analytic Hilbert module $\mathcal H$ on $\Omega$ such that $\mathcal M = [\mathcal I]$, closure of the ideal $\mathcal I$ in $\mathcal H.$ It is known that $\mathbb V_0(\mathcal I)= \mathbb V_0(\mathcal M)$ (cf. \cite{bmp, dpsy}). Since 
$$
\mathcal M\subseteq\mathcal M_0^e:=\{f\in\mathcal H: q(D)f|_0=0 \mbox{~for~all~}q\in\mathbb V_0(\mathcal M)\},
$$ 
it follows that 
\beqa
\mbox{dim}\mathcal H/\mathcal M_0^e\leq\mbox{dim}\mathcal
H/\mathcal M = \dim \C[\underline{z}]/\mathcal I &\leq& \dim \C[\underline{z}]/\mathfrak m_0^N\\ &\leq& \sum_{k=0}^{N-1}\binom{k+m-1}{m-1} <+\infty.
\eeqa
Therefore,  from \cite{dpsy}, we have $\mathcal M_0^e\cap\C[\underline z]=\mathcal I_0^e$ and $\mathcal M\cap\C[\underline z]=\mathcal I$,  and hence \beq \label{obs1} \mathcal
M_0^e=[\mathcal I_0^e]=[\mathcal I]=\mathcal M. \eeq

\noindent\textsf{Assumption}: Let $\mathcal I \subseteq \C[\underline z]$ be an ideal. We assume that the module  $\mathcal M$ in $\mathfrak B_1(\Omega)$ is the completion of $\mathcal I$ with respect to some inner product. For notational convenience, in the following discussion, we let $K$ be the reproducing kernel of $\mathcal M = [\mathcal I]$, instead of $K_{[\mathcal I]}$.

To describe the joint kernel $\cap_{j=1}^m\ker (M_j -w_j )^*$ using the characteristic space $\mathbb V_w(\mathcal I)$, it will be useful to define the auxialliary space 
$$
\tilde{\mathbb V}_{w}(\mathcal I)~=~\{q\in\C[\underline z]: \frac{\partial q}{\partial z_i}\in
\mathbb V_{w}(\mathcal I),\: 1\leq i\leq m\}.
$$ 
From \cite[Lemma 3.4]{bmp}, it follows that $V(\mathfrak m_w\mathcal I)\setminus V(\mathcal I) =\{w\}$ and $\mathbb V_w(\mathfrak m_w\mathcal I) = \tilde{\mathbb V}_{w}(\mathcal I)$.  Therefore,
\beq  \label{crucial}
\dim \cap_{j=1}^m\ker (M_j -w_j )^*&=& \dim \mathcal M/\mathfrak m_w\mathcal M   ~=~\dim \mathcal I/\mathfrak m_w\mathcal I \\ &=&\sum_{\lambda\in V(\mathfrak m_w\mathcal I)\setminus V(\mathcal I)} \dim \mathbb V_\lambda(\mathfrak m_w\mathcal I)/\mathbb V_\lambda(\mathcal I) \nonumber \\ &=& \dim \tilde{\mathbb V}_{w}(\mathcal I)/\mathbb V_w(\mathcal I)\nonumber.
\eeq
For the second and the third equalities, see \cite[Theorem 2.2.5 and 2.1.7]{cg}. Since $\tilde{\mathbb V}_{w}(\mathcal I)$ is a subspace of the inner product space $\C[\underline{z}]$, we will often identify the quotient space $\tilde{\mathbb V}_{w}(\mathcal I)/\mathbb V_w(\mathcal I)$ with the subspace of $\tilde{\mathbb V}_{w}(\mathcal I)$ which is the orthogonal complement of $\mathbb V_w(\mathcal I)$ in $\tilde{\mathbb V}_{w}(\mathcal I)$. Equation \eqref{crucial} motivates following lemma describing the basis of the joint kernel of the adjoint of the multiplication operator at a point in $\Omega$. This answers the question (\ref{q1}) of the introduction.

\begin{lem}\label{nice}
Fix $w_0\in\Omega$ and polynomials $q_1,\ldots,q_t$. Let $\mathcal I$ be a polynomial ideal and $K$ be the reproducing kernel corresponding the Hilbert module $[\mathcal I]$, which is assumed to be in $\mathfrak B_1(\Omega^*)$. Then the vectors 
$$
q^*_1(\bar D)K(\cdot, w)|_{w=w_0},\ldots, q^*_t(\bar D)K(\cdot, w)|_{w=w_0}
$$ 
form a basis of the joint kernel at $w_0$ of the adjoint of the multiplication operator
if and only if the classes $[q_1],\ldots,[q_t]$ form a basis of $\tilde{\mathbb V}_{w_0}(\mathcal I)/\mathbb V_{w_0}(\mathcal I)$.
\end{lem}
\begin{proof} Without loss of generality we assume $0\in\Omega$ and $w_0=0$.

\textsf{Claim 1}: For any $q\in \C[\underline z]$, the vector $q^*(\bar D)K(\cdot,w)|_{w=0}\neq 0$ if and only if
$q\notin\mathbb V_0(\mathcal I)$.

Using the reproducing property $f(w) = \langle f, K(\cdot,w)\rangle$  of the kernel $K$,
it is easy to see (cf. \cite{cs}) that
$$\partial^{\alpha}f(w) = \langle f,
\bar\partial^{\alpha} K(\cdot,w)\rangle, \mbox{~for~}  \alpha\in\mathbb Z_m^+, ~w\in\Omega,
 ~f\in\mathcal M.
$$
and thus
\begin{eqnarray*}
\partial^{\alpha}f(w)|_{w=0}&=&\langle f, \bar\partial^{\alpha}
K(\cdot,w)\rangle|_{w=0}~=~\langle f, \bar\partial^{\alpha} \{\d_{\beta}
\frac{{\partial}^{\beta}K(z,0)}{{\beta}!}{\bar w}^{\beta}\}\rangle|_{w=0}\\ &=& \langle f,
\{\d_{\beta\geq\alpha} \frac{{\partial}^{\beta}K(z,0){\alpha}!}{{\beta}!}{\bar
w}^{\beta-\alpha}\}\rangle|_{w=0}~=~\{\d_{\beta\geq\alpha} \langle f,
\frac{{\partial}^{\beta}K(z,0){\alpha}!}{{\beta}!}\rangle{\bar w}^{\beta-\alpha}\}|_{w=0}\\ &=& \langle f,
\bar\partial^{\alpha} K(\cdot,w)|_{w=0}\rangle.
\end{eqnarray*}
So for $f\in\mathcal M$ and a polynomial $q=\d a_{\alpha}z^{\alpha}$, we have
\begin{eqnarray}\label{gendiff}
\langle f, q^*(\bar D)K(\cdot, w)|_{w=0}\rangle &=& \langle q,
\d_{\alpha} {\bar a}_{\alpha}\bar\partial^{\alpha}K(\cdot, w)\rangle|_{w=0}= \d_{\alpha}  a_{\alpha}\langle f,
\bar\partial^{\alpha}K(\cdot, w)\rangle|_{w=0}  \\
&=& \{\d_{\alpha}  a_{\alpha}\partial^{\alpha}\langle f,
K(\cdot, w)\rangle\}|_{w=0} =q(D)f|_{w=0}.\nonumber
\end{eqnarray}
This proves the claim.

\textsf{Claim 2}: For any $q\in\C[\underline z]$, the vector $q^*(\bar D)K(\cdot,w)|_{w=0}\in\cap_{j=1}^m\ker M_j ^*$ if and only if $q\in\tilde{\mathbb V}_{0}(\mathcal I)$.

For any $f\in\mathcal M$, we have
\beqa
\langle f, M_j^*q^*(\bar D)K(\cdot,w)|_{w=0}\rangle &=& \langle M_jf, q^*(\bar D)K(\cdot,w)|_{w=0}\rangle = q(D)(z_jf)|_{w=0} \\ &=& \{z_jq(D)f + \frac{\partial q}{\partial z_j}(D)f\}|_{w=0} = \frac{\partial q}{\partial z_j}(D)f|_{w=0}
\eeqa
verifying the claim.  

As a consequence of claims 1 and 2, we see that  $q^*(\bar D)K(\cdot,w)|_{w=0}$ is a non-zero vector in the joint kernel if and only if the class $[q]$ in $\tilde{\mathbb V}_{0}(\mathcal I)/\mathbb V_0(\mathcal I)$ is non-zero.

Pick polynomials $q_1,\ldots,q_t$. From the equation \eqref{crucial} and claim 2, it is enough to show that  $q^*_1(\bar D)K(\cdot, w)|_{w=0},\ldots, q^*_t(\bar D)K(\cdot, w)|_{w=0}$ are linearly independent if and only if $[q_1],\ldots,[q_t]$ are linearly independent in $\tilde{\mathbb V}_{0}(\mathcal I)/\mathbb V_{0}(\mathcal I)$.  But from claim 1 and equation \eqref{gendiff}, it follows that 
$$
\sum_{i=1}^t\bar\alpha_iq^*_i(\bar D)K(\cdot,w)|_{w=0} = 0\mbox{~if ~and ~only ~if~} \sum_{i=1}^t\alpha_i[q_i] =0 \mbox{~in~} \tilde{\mathbb V}_{0}(\mathcal I)/\mathbb V_{0}(\mathcal I)
$$ 
for scalars $\a_i\in\C,\,1\leq i\leq t$. This completes the proof.
\end{proof}


\begin{rem}
The `if' part of the theorem can also be obtained from the decomposition theorem \cite[Theorem 1.5]{bmp}. For module $\mathcal M$ in the class $\mathfrak B_1(\Omega^*)$,
let $\mathcal S^{\mathcal M}$ be the subsheaf of the sheaf of holomorphic functions $\mathcal O_\Omega$ whose stalk $\mathcal S^{\mathcal M}_w$ at $w \in
\Omega$ is
$$ \big \{(f_1)_w \mathcal O_w + \cdots + (f_n)_w \mathcal
O_w : f_1, \ldots , f_n \in \mathcal M \big \}, $$ 
and the characteristic space at $w\in\Omega$  
is the vector space 
$$
\mathbb V_w(\mathcal S_w^\mathcal M) ~=~ \{q\in\C[\underline z]:q(D)f\big |_w=0,\: f_{w}\in \mathcal S^{\mathcal M}_{w}\}.
$$ 
Since $$\dim\mathcal S^\mathcal M_{0}/\mathfrak
m_{0} \mathcal S^\mathcal M_{0} = \dim\cap_{j=1}^m\ker M_j ^*= \dim \tilde{\mathbb V}_{0}(\mathcal I)/\mathbb V_0(\mathcal I) = t,$$ there exists a minimal set of generators $g_1,\cdots,g_t$ of $\mathcal S^\mathcal M_{0}$ and  a $r>0$ such that
$$
K(\cdot, w)=\sum_{i=1}^t\ov{g_j(w)}K^{(j)}(\cdot,w) \mbox{~for~all~}w\in\Delta(0; r)
$$
for some choice of anti-holomorphic functions $K^{(1)}, \ldots , K^{(t)}:\Delta(0; r) \to\mathcal M$. The formula 
\begin{eqnarray}\label{gf}
q(D) (z^\alpha g)= \sum_{k\leq\alpha} \binom{\alpha}{k} z^{\alpha-k}\frac{\partial^kq}{\partial z^k}(D)(g)
\end{eqnarray}
gives 
$$
q_i^*(\bar D)K(\cdot,w)|_{w=0}=\sum_{j=1}^t\{ K^{(j)}(\cdot, w)|_{w=0}\}\{q_i^*(\bar D)\ov{g_j(w)}|_{w=0}\}
$$
for $q_i\in\tilde{\mathbb V}_{0}(\mathcal I),\,1\leq i\leq t$. The proof follows from the fact that $\mathbb V_w(\mathcal I)= \mathbb V_w(\mathcal M) =\mathbb V_w(\mathcal S^{\mathcal M}_w)$.
\end{rem}

\begin{rem}
We give details of the case where the ideal $\mathcal I$ is
singly generated, namely $\mathcal I = <p>$. From \cite{dmv}, it
follows that the reproducing kernel $K$ admits a global
factorization, that is, $K(z,w) = p(z)\chi(z,w)\bar p(w)$ for
$z,w\in\Omega$ where $\chi(w,w)\neq 0$ for all $w\in\Omega$. So we
get $K_1(\cdot, w) = p(\cdot)\chi(\cdot,w)$ for all $w\in\Omega$.
The proposition above gives a way to write down this section in term
of reproducing kernel. Let $0\in V(\mathcal I)$. Let $q_0$ be the
lowest degree term in $p$. We claim that $[q_0^*]$ gives a
non-trivial class in $ \tilde{\mathbb V}_{0}(\mathcal I)/\mathbb
V_0(\mathcal I) $. This is because all partial derivatives of
$q_0^*$ have degree less than that of $q^*_0$ and hence from
\eqref{gf}
$$
q_0^*(D) (z^\alpha g)|_0  = \frac{\partial^\alpha q_0^*}{\partial z^\alpha}(D)(p)\big{|}_0 = 0 \mbox{~for~all ~multi-indices~}\alpha \mbox{~such~that~} |\alpha|>0
$$ and thus $\frac{\partial q_0^*}{\partial z_i}\in \mathbb V_0(\mathcal I) $ for all $i,\, 1\leq i\leq m$, that is, $q_0^*\in\tilde{\mathbb V}_{0}(\mathcal I)$. Also as the lowest degree of $p-q_0$ is strictly greater than that of $q_0$,
$$
q_0^*(D)p|_0 = q_0^*(D)(p- q_0+q_0)|_0 = q_0^*(D)q_0|_0 = \parallel q_0\parallel_0^2>0
$$
This shows that $q_0^*\notin \mathbb V_0(\mathcal I)$ and hence gives a non-trivial class in $ \tilde{\mathbb V}_{0}(\mathcal I)/\mathbb V_0(\mathcal I) $. Therefore from the proof of Lemma \ref{nice}, we have
$$
q_0(\bar D)K(\cdot, w)|_{w=0} = K_1(\cdot, w)|_{w=0}q_0(\bar D)\ov {p(w)}|_0 = \parallel q_0^*\parallel_0^2 K_1(\cdot, w)|_{w=0}.
$$
Let $q_{w_0}$ denotes the lowest degree term in $z - w_0$ in the expression of $p$ around $w_0$. Then we can write
\begin{eqnarray}\label{sl}
K_1(\cdot, w)|_{w=w_0} =
\begin{cases}
\frac{K(\cdot, w)|_{w=w_0}}{p(w_0)} &\: \mbox{if}\: w_0\notin V(\mathcal I)\cap\Omega\\
\frac{q_{w_0}(\bar D)K(\cdot, w)|_{w=w_0}}{\parallel q_{w_0}^*\parallel_{w_0}^2}&\:\mbox{if}\: w_0\in V(\mathcal I)\cap\Omega.
\end{cases}
\end{eqnarray}
\end{rem}

For a fixed set of polynomials $q_1,\ldots, q_t$, the next lemma provides a sufficient condition for the classes $[q_1^*],\ldots,\,[q_t^*]$ to be linearly independent in 
$\tilde{\mathbb V}_{w_0}(\mathcal I)/\mathbb V_{w_0}(\mathcal I)$. The ideas involved in the two easy but different proofs given below will be used repeatedly in the sequel.

\begin{lem}\label{sc}
Let $q_1,\ldots, q_t$ are linearly independent polynomials in the polynomial ideal $\mathcal I$ such that $q_1,\ldots, q_t\in \tilde{\mathbb V}_{0}(\mathcal I)$. Then $[q_1^*],\ldots,[q_t^*]$ are linearly independent in 
$\tilde{\mathbb V}_{w_0}(\mathcal I)/\mathbb V_{w_0}(\mathcal I)$.
\end{lem}
\begin{proof}[First Proof]
Suppose $\sum_{i=1}^t \alpha_i [q_i^*] = 0$ in $\tilde{\mathbb V}_{w_0}(\mathcal I)/\mathbb V_{w_0}(\mathcal I)$ for some $\alpha_i\in\C,\,1\leq i\leq t$. Thus $\sum_{i=1}^t \alpha_i q_i^* = q$ for some $q\in \mathbb V_{w_0}(\mathcal I)$. Taking the inner product of $\sum_{i=1}^t \alpha_i q_i^*$ with $q_j$ for a fixed j, we get
$$
\sum_{i=1}^t \langle q_j, q_i \rangle_{w_0} =\big( \sum_{i=1}^t \alpha_i q_i^*\big)(D)q_j|_{w_0}= q(D)q_j|_{w_0}= 0.
$$
The Grammian $\big(\!(\langle q_j, q_i \rangle_{w_0})\!\big)_{i,j=1}^t$ of the linearly independent polynomials $q_1,\ldots, q_t$ is non-singular. Thus $\alpha_i=0,\,1\leq i\leq t$, completing the proof.

\noindent\textit{Second Proof}. If $[q_1^*],\ldots,[q_t^*]$ are not linearly independent, then we may assume without loss of generality that $[q_1^*] =\sum_{i=2}^t\alpha_i [q_i^*]$ for $\alpha_1,\ldots,\alpha_t\in\C$. Therefore $[q_1^*  - \sum_{i=2}^t\alpha_i p_i^*] = 0$ in the quotient space $\tilde{\mathbb V}_{w_0}(\mathcal I)/\mathbb V_{w_0}(\mathcal I)$, that is, $q_1^*  - \sum_{i=2}^t\alpha_i q_i^* \in \mathbb V_{w_0}(\mathcal I)$. So, we have
$$
(q_1^*  - \sum_{i=2}^t\alpha_i q_i^*)(D)q|_{w_0}=0 \mbox{~for~all~} q\in\mathcal I.
$$
Taking $q=q_1  - \sum_{i=2}^t\bar\alpha_i q_i $ we have $\parallel q_1  - \sum_{i=2}^t\bar\alpha_i q_i\parallel_{w_0}^2 = 0$. Hence $q_1  = \sum_{i=2}^t\bar\alpha_i q_i$ which is a contradiction.
\end{proof}

Suppose are $p_1,...,p_t$ are a minimal set of generators for $\mathcal I$. Let $\mathcal M$ be the completion of $\mathcal I$ with respect to some inner product induced by a positive definite kernel. We recall from \cite{dp} that rank$_\C[\underline z]\mathcal
M=t$. Let $w_0$ be a fixed but arbitrary point in $\Omega$. We ask if there exist a choice of generators $q_1,...,q_t$ such that $q^*_1(\bar D)K(\cdot,w)_{0},\ldots, q^*_t(\bar D)K(\cdot,w)_{0}$ forms a basis for $\cap_{j=1}^m\ker (M_j -w_{0j}) ^*$. We isolates some instances where the answer is affirmative. However, this is not always possible (see remark \ref{ex}). From \cite[Lemma 5.11, Page-89]{dp}, we have
$$
\mbox{dim}\cap_{j=1}^m\ker M_j*=\mbox{dim}\mathcal M/\mathfrak m_0\mathcal
M=\mbox{dim}\mathcal M\otimes_{\C[\underline z]}\C_0\leq \mbox{rank}_{\C[\underline z]}\mathcal M.\mbox{dim}\C_0\leq t,
$$
where $\mathfrak m_0$ denotes the maximal ideal of $\C[\underline z]$ at $0$.  So we have $\dim \cap_{j=1}^m\ker M^*_j\leq t$. The germs $p_{10},\ldots, p_{t0}$ forms a set of generators, not necessarily minimal, for ${\mathcal S}^{\mathcal
M}_0$. However minimality can be assured under some additional hypothesis.  For example, let $\mathcal I$ be the ideal  generated by the polynomials $z_1(1+z_1), z_1(1 - z_2), z_2^2$. This is minimal set of generators for the ideal $\mathcal I$, hence for $\mathcal M$, but not for ${\mathcal S}^{\mathcal M}_0$. Since $\{z_1,z_2\}$ is a minimal set of generators for ${\mathcal S}^{\mathcal
M}_0$, it follows that $\{z_1(1+z_1), z_1(1 - z_2), z_2^2\}$ is not minimal for ${\mathcal S}^{\mathcal
M}_0$. This was pointed out by R. G. Douglas.

\begin{lem}\label{jl}
Let $p_1,\ldots,p_t$ be homogeneous polynomials, not necessarily of the same degree. Let $\mathcal I\subset \C[\underline z]$ be an ideal for which $p_1,\ldots,p_t$ is a minimal set of generators. Let $\mathcal M$ be a submodule of an analytic Hilbert module over $\C[\underline z]$ such that $\mathcal M = [\mathcal I]$. Then the germs $p_{10},\ldots, p_{t0}$ at $0$ forms a minimal set of generators for ${\mathcal S}^{\mathcal M}_0$.
\end{lem}
\begin{proof}
For $1\leq i\leq t$, let deg $p_i = \alpha_i$. Without loss of generality we assume that $\alpha_i\leq \alpha_{i+1} ,\, 1\leq i\leq t-1$. Suppose the germs $p_{10},\ldots, p_{t0}$ are not minimal, that is, there exist $k ( 1\leq k\leq t)$, $p_k=\sum_{i=1,i\neq k}^t\phi_ip_i$ for some choice of holomorphic functions $\phi_i,\,1\leq i\leq t, i\neq k$ defined on a suitable small neighborhood of $0$. Thus we have
$$
p_k = \sum_{i:\alpha_i\leq \alpha_k}\phi_i^{\alpha_k-\alpha_i}p_i,
$$
where $\phi_i^{\alpha_k-\alpha_i}$ is the Taylor polynomial containing  of $\phi_i$ of degree  $\alpha_k-\alpha_i$. Therefore $p_1,\ldots,p_t$ can not be a minimal set of generators for the ideal $\mathcal I$. This contradiction completes the proof.
\end{proof}

Consider the ideal $\mathcal I$ generated by the polynomials $z_1+z_2+z_1^2, z_2^3-z_1^2$. We will see later that the joint kernel at $0$, in this case is spanned by the independent vectors $p(\bar D)K(\cdot,w)|_{w=0}, q(\bar D)K(\cdot,w)|_{w=0}$, where $p=z_1+z_2$ and $q=(z_1 - z_2)^2$. Therefore any vectors in the joint kernel is of the form $(\alpha p+\beta q)(\bar D)K(\cdot,w)|_{w=0}$ for some $\alpha,\beta\in\C$. It then follows that $\alpha p+\beta q$ and $\alpha' p+\beta' q$ can not be a set of generators of $\mathcal I$ for any choice of $\alpha,\beta,\alpha',\beta'\in\C$. However in certain cases, this is possible. We describe below the case where $\{p_1(\bar D)K(\cdot,w)|_{w=0},...,p_t(\bar
D)K(\cdot,w)|_{w=0}\}$ forms a basis for $\cap_{j=1}^n \ker M_j^*$ for an obvious choice of generating set in $\mathcal I$.

\begin{lem}\label{kth} Let $p_1,\ldots,p_t$ be homogeneous polynomials of same degree. Suppose that $\{p_1,\ldots,p_t\}$ is a minimal set of generators for the ideal $\mathcal I\subset\C[\underline{z}]$.  Then the set 
$$
\{p_1(\bar D)K(\cdot,w)|_{w=0},...,p_t(\bar
D)K(\cdot,w)|_{w=0}\}
$$ 
forms a basis for $ \displaystyle\cap_{j=1}^n \ker M_j^*$.
\end{lem}
\begin{proof}
For $1\leq i\leq t$, let deg $p_i=k$. It is enough to show, using Lemma \ref{nice}, \ref{sc} and \ref{jl}, that the polynomials $p_1^*,\ldots,p_t^*$ are in $\tilde{\mathbb V}_{0}(\mathcal I)$. Since $\frac{\partial p_i^*}{\partial z_j}$ is of degree at most $k-1$ for each $i$ and $j,\,1\leq i\leq t,\,1\leq j\leq m$, and the the term of lowest degree in each polynomial in the ideal $p\in\mathcal I$ will be at least of degree $k$, it follows that $\frac{\partial p_i^*}{\partial z_j}(D)p|_0=0,\,p\in\mathcal I,\,1\leq i\leq t,1\leq j\leq m$. This completes the proof.
\end{proof}

\begin{ex}
Let $\mathcal M$ be an analytic Hilbert module over
$\Omega\subseteq{\C}^m$, and ${\mathcal M}_{n}$ be a submodule of
$\mathcal M$ formed by the closure of polynomial ideal $\mathcal I$
in $\mathcal M$ where $\mathcal
I= \langle z^{\alpha}=z_1^{\alpha_1}...z_m^{\alpha_m}:\alpha_i\in\N\cup\{0\},|\alpha|=\d_{i=1}^m\alpha_i=n  \rangle$.
We note that $Z(\tau)=\{0\}$. Let $K_{n}$ be the reproducing
kernel corresponding to ${\mathcal M}_{n}$. Then,
\begin{enumerate}
\item[(1)] ${\mathcal M}_{n}=\{f\in\mathcal M: \partial^{\alpha}f(0)=0, \mbox{~for~}\alpha_i\in\N\cup\{0\},|\alpha|\leq n-1\}$
\item[(2)] $\bigcap_{j=1}^m\ker(M_j|_{{\mathcal M}_{n}}-w_j)^*= \left\{
\begin{array}{ll}
\rm{span}\{K_{n}(\cdot,w)\}, & \hbox{for $w\neq 0$;} \\
\rm{span}\{\bar\partial^{\alpha}K_{n}( \cdot,
$w$)|_{w=0}:\alpha_i\in\N\cup\{0\},|\alpha|=n\}, & \hbox{for
$w=0$.}
\end{array}
\right. $
\end{enumerate}
\end{ex}

We now go further and show that a similar description of the joint kernel is possible even if the restrictive assumption of ``same degree" is removed. We begin with the simple case of two generators.

\begin{prop}\label{gen}
Suppose $\{p_1, p_2\}$ is a minimal set of generators for the ideal
$\mathcal I$. and are homogeneous with deg $p_1\neq$ deg $p_2$. Let $K$ be the reproducing kernel corresponding the Hilbert module $[\mathcal I]$, which is assumed to be in $\mathfrak B_1(\Omega^*)$. Then there exist polynomials $q_1,q_2$ which generate
the ideal $\mathcal I$ and 
$$
\{q_1(\bar
D)K(\cdot,w)|_{w=0},\, q_2(\bar D)K(\cdot,w)|_{w=0}\}
$$ 
is a basis for $\cap_{j=1}^m\ker M_j^*$.
\end{prop}
\begin{proof}
Let deg $p_1=k$ and deg $p_2=k+n$ for some $n\geq 1$. The set $\{p_1, p_2+(\sum_{|i|=n}\gamma_iz^i)p_1\}$ is a minimal set of generators for $\mathcal I,\, \g_i\in\C$ where $i = (i_1,\ldots,i_m)$ and $|i| =
i_1+\ldots+i_m$. We will take $q_1=p_1$ and find constants $\g_i$ in $\C$ such that
$$
q_2 =p_2+ (\sum_{|i|=n}\gamma_iz^i)p_1.
$$
We have to show (Lemma \ref{nice}) that $\{[q_1^*],[q_2^*]\}$ is a basis in $\tilde{\mathbb V}_{0}(\mathcal I)/\mathbb V_0(\mathcal I)$.
From the equation \eqref{crucial} and Lemma \ref{sc}, it is enough to show that $q_2^*$ is a  in $\tilde{\mathbb V}_{0}(\mathcal I)$. To ensure that $\frac{\partial q_2^*}{\partial z_k}\in\mathbb V_0(\mathcal I),\,1\leq k\leq m$, we need to check: 
$$
\frac{\partial^{|\alpha|}q_2^*}{\partial z^{\alpha}}(D)p_i|_{w=0} = \langle p_i,  \frac{\partial^{|\alpha|}q_2}{\partial z^{\alpha}}\rangle|_0 = 0, 
$$
for all multi-index $\alpha =(\alpha_1,\ldots,\alpha_m)$  with 
$1\leq|\alpha|\leq n$ and $i=1,2$. For $|\a|>n$, these conditions are evident. Since the degree of the polynomial $q_2$ is $k+n$, we have $\langle p_2,  \frac{\partial^{|\alpha|}q_2}{\partial z^{\alpha}}\rangle_0 = 0,\,1\leq|\alpha|\leq n$. If $n>1$, then $\langle p_1,  \frac{\partial^{|\alpha|}q_2}{\partial z^{\alpha}}\rangle_0 = 0,\,1\leq|\alpha|< n$. To find $\gamma_{i},\,i = (i_1,\ldots,i_m)$, we solve the equation $\langle p_1,  \frac{\partial^{|\alpha|}q_2}{\partial z^{\alpha}}\rangle|_0 = 0$ for all $\alpha$ such that $|\alpha|=n$. By the Leibnitz rule,
\beqa
\frac{\partial^{|\alpha|}q_2^*}{\partial z^{\alpha}} &=& \frac{\partial^{|\alpha|}p_2^*}{\partial z^{\alpha}}+\sum_{\nu\leq \alpha}\binom{\alpha}{\nu}\partial^{\alpha -\nu}(\sum_{|i|=n} \bar \g_iz^i)\frac{\partial^{|\nu|}p_1^*}{\partial z^{\nu}}   \\ &=& \frac{\partial^{|\alpha|}p_2^*}{\partial z^{\alpha}}+ \sum_{\nu\leq \alpha}\binom{\alpha}{\nu}(\sum_{|i|=n,i\geq \alpha - \nu} \bar \g_i\frac{i!}{(i-\alpha+\nu)!}z^{i-\alpha+\nu})\frac{\partial^{|\nu|}p_1^*}{\partial z^{\nu}} .
\eeqa
Now $\frac{\partial^{|\alpha|}p^*}{\partial z^{\alpha}}(D)p_i|_{w=0}=0$ gives
\beq\label{rew}
 0&=&\Big{(}\frac{\partial^{|\alpha|}p_2^*}{\partial z^{\alpha}}+ \sum_{\nu\leq \alpha}\binom{\alpha}{\nu}\Big{(}\sum_{|i|=n,i\geq \alpha - \nu} \bar \g_i\frac{i!}{(i-\alpha+\nu)!}z^{i-\alpha+\nu}\Big{)}\frac{\partial^{|\nu|}p_1^*}{\partial z^{\nu}}\Big{)}(D)p_1|_{w=0}
\\ &=& \langle p_1, \frac{\partial^{|\alpha|}p_2}{\partial z^{\alpha}}\rangle_0+
\sum_{r=0}^n\sum_{|i|=n}\ov{A_{\a i}(r)}\bar \g_i,\nonumber
\eeq
where given the multi-indices $\alpha, i$,
\beq\label{le}
A_{\a i}(r)~= \begin{cases}\sum_{\nu} \binom{\alpha}{\nu}\frac{i!}{(i-\alpha+\nu)!}\langle \frac{\partial^{|\nu|}p_1}{\partial z^{\nu}}, \frac{\partial^{|i-\alpha+\nu|}p_1}{\partial z^{i-\alpha+\nu}} \rangle_0 & |\nu|=r,\nu\leq \alpha, i\geq \alpha - \nu;\\ 0 & \mbox{otherwise}.\end{cases}
\eeq
Let  $A(r) = \big(\!\big(A_{\a i}(r)\big)\!\big) $ be the  $\tbinom{n+m-1}{m-1}\times \tbinom{n+m-1}{m-1}$ matrix in colexicographic order on $\alpha$ and $i$.
Let $A = \sum_{r=0}^n A(r)$ and $\g_n$ be the $\tbinom{n+m-1}{m-1}\times 1$ column vector $(\g_i)_{|i|=n}$. Thus the equation \eqref{rew} is of the form
\beq\label{lineq}
\bar A\bar\g_n= \Gamma ,
\eeq
where $\Gamma$ is the $\binom{n+m-1}{m-1}\times 1$ column vector $(-\langle p_1, \frac{\partial^{|\alpha|}p_2}{\partial z^{\alpha}}\rangle_0)_{|\a|=n}$. Invertibility of the coefficient matrix $A$ then guarantees the existence of a solution to the equation \eqref{lineq}. We show that the matrix $A(r)$ is non-negative definite and the matrix $A(0)$ is diagonal:
\beq\label{pd}
A(0)_{\a i}= \begin{cases} \a !\parallel p_1\parallel^2
&\mbox{if}\,\, \a=i\\ 0  & \mbox{if}\,\,\a\neq i. \end{cases}
\eeq
and therefore positive definite. Fix a $r, \,1\leq r\leq n$. To prove that $A(r)$ is non-negative definite, we show that it is the Grammian with respect to Fock inner product at $0$. To each $\mu= (\mu_1,\ldots,\mu_m)$ such that $|\mu| = n-r$, we associate a $1\times\binom{n+m-1}{m-1}$ tuple of polynomials $X_{\mu}^{r}$, defined as follows
$$
X^{r}_\mu(\beta) = \begin{cases}\mu !\binom{\beta}{\beta-\mu}\frac{\partial^{|\beta-\mu|}p_1}{\partial z^{\beta-\mu}} &\mbox{if}\,\, \beta\geq \mu\\ 0 &\mbox{otherwise,}\end{cases}
$$
where $\beta= (\beta_1,\ldots,\beta_m),\,|\beta|=n$ ($\beta\geq \mu$ if and only if $\beta_i\geq\mu_i$ for all $i$). By $X^{r}_{\mu}\cdot (X^{r}_{\mu})^t$, we denote the $\binom{n+m-1}{m-1}\times \binom{n+m-1}{m-1}$ matrix whose $\a i$-th element is $\langle X^{r}_\mu(\a), X^{r}_\mu(i)\rangle_0, \, |\a|=n=|i|$. We note that
\beq\label{pd1}
\sum_{|\mu|=n-r}\frac{1}{\mu!}(X^{r}_\mu\cdot (X^{r}_\mu)^t)_{\a i} &=& \sum_{|\mu|=n-r}\frac{1}{\mu!}\langle X^{r}_\mu(\a), X^{r}_\mu(i)\rangle_0 \\ &=& \sum_{|\mu|= n-r,\a\geq \mu, i\geq\mu }\frac{1}{\mu!}\langle \mu !\binom{\a}{\a-\mu} \frac{\partial^{|\a-\mu|}p_1}{\partial z^{\a-\mu}}, \mu !\binom{i}{i-\mu} \frac{\partial^{|i-\mu|}p_1}{\partial z^{i-\mu}} \rangle_0\nonumber\\ &=& \sum_{|\nu|= r,\nu\leq\a, i\geq \a -\nu }(\a -\nu)!\binom{\a}{\nu}\binom{i}{i-\a+\nu}\langle  \frac{\partial^{|\a-\mu|}p_1}{\partial z^{\a-\mu}}, \frac{\partial^{|i-\mu|}p_1}{\partial z^{i-\mu}} \rangle_0 \nonumber \\
&=& A_{\a i}(r). \nonumber
\eeq
Since $X^{r}_\mu\cdot (X^{r}_\mu)^t$ is the Grammian of the vector tuple $X^{r}_\mu$, it is non-negative definite. Hence $A(r) = \sum_{|\mu|=n-r}\frac{1}{\mu!}(X^{r}_\mu\cdot (X^{r}_\mu)^t)$ is non-negative definite. Therefore $A$ is positive definite and hence equation \eqref{lineq} admits a solution, completing the proof.
\end{proof}

Let $\mathcal I$ be a homogeneous polynomial ideal. As one may expect, the proof in the general case is considerably more involved. However the idea of the proof is similar to the simple case of two generators. Let $p_1,\ldots,p_v$ be a minimal set of generators, consisting of homogeneous polynomials, for the ideal $\mathcal I$. We arrange the set $\{p_1,\ldots, p_v\}$ in blocks of polynomials $P^1,\ldots,P^k$ according to ascending order of their degree, that is,
$$
\{P^1,\ldots,P^k\} ~=~ \{ p_1^1,\ldots,p_{u_1}^1,p_1^2,\ldots,p_{u_2}^2,\ldots,p_1^l,\ldots,p_{u_l}^l,\ldots,p_1^k,\ldots,p_{u_k}^k \},
$$ 
where each $P^l = \{p_1^l,\ldots,p_{u_l}^l\},\, 1\leq l\leq k$ consists of homogeneous polynomials of the same degree, say $n_l$ and
$n_{l+1}>n_l,\,1\leq l\leq k-1.$ 
As before, for $l=1$, we take $q_j^1=p_j^1,\, 1\leq j\leq u_1$ and for $l\geq 2$ take
$$
q_j^l = p_j^l+\sum_{f=1}^{l-1}\sum_{s=1}^{u_f}\g_{lj}^{fs}p_s^f,
\mbox{~where~}
\g_{lj}^{fs}(z) ~=~\sum_{|i|=n_l-n_f} \g_{lj}^{fs}(i)z^i.
$$ 
Each $\g_{lj}^{fs}$ is a polynomial of degree $n_l-n_f$ for some choice of $\g_{lj}^{fs}(i)$ in $\C$. So we obtain another set of polynomials  $\{Q^1,\ldots,Q^k\}$ with $Q^l = \{q_1^l,\ldots,q_{u_l}^l\},\, 1\leq l\leq k$ satisfying the the same property as the set of polynomials $\{P^1,\ldots,P^k\}$. From Lemma \ref{nice} and \ref{sc}, it is enough to check $q_j^{l*}$ is in $\tilde{\mathbb V}_{0}(\mathcal I)$. This condition yields a linear system of equation as in the proof of Proposition \ref{gen}, except that the co-efficient matrix is a block matrix with each block similar to $A$ defined by the equation \eqref{le}. For $q_j^{l*}$ in $\tilde{\mathbb V}_{0}(\mathcal I)$, the constants  $\g_{lj}^{fs}(i)$ must satisfy:
\beqa\label{m}
\lefteqn{0 = \frac{\partial^{|\a|}q_j^{l*}}{\partial z^{\a}}(D)p_t^e|_0}\\  &=& \langle p_t^e,\frac{\partial^{|\a|}p_j^l}{\partial z^{\a}}\rangle_0 + \sum_{f=1}^{l-1}\sum_{s=1}^{u_f}\sum_{\nu\leq \a}\binom{\a}{\nu}\sum_{|i|=n_l-n_f, i\geq \a-\nu} \ov{\g_{lj}^{fs}(i)}\frac{i!}{(i-\alpha+\nu)!} \langle \frac{\partial^{|i-\a+\nu|}p_t^e}{\partial z^{i-\a+\nu}}, \frac{\partial^{|\nu|}p_s^f}{\partial z^{\nu}}\rangle_0
\eeqa
All the terms in the equation are zero except when $|\a| = n_l - n_d,\, 1\leq d\leq l-1$. For $e=d=f$, we have the equations
\beq\label{m1}
-\langle p_t^d,\frac{\partial^{|\a|}p_j^l}{\partial z^{\a}}\rangle_0~=~\sum_{s=1}^{u_d}\!\sum_{r=0}^{n_l-n_d}\!\sum_{|i|=n_l-n_d}\! \ov{\big(A_{st}^d(r)\big)_{\alpha i}}\ov{\g_{lj}^{ds}(i)},
\eeq
where
$$
\big{(}A_{st}^d(r)\big{)}_{\a i}~=\begin{cases} \sum_{\nu} \binom{\alpha}{\nu}\frac{i!}{(i-\alpha+\nu)!}\langle \frac{\partial^{|\nu|}p_s^d}{\partial z^{\nu}},\frac{\partial^{|i-\a+\nu|}p_t^d}{\partial z^{i-\a+\nu}} \rangle_0 & |\nu|=r,\nu\leq \alpha, i\geq \alpha - \nu;\\ 0 & \mbox{otherwise}.\end{cases}
$$
Let $A_{st}^d(r)$ be the $\binom{n_l-n_{d-1}+m-1}{m-1}\times \binom{n_l-n_{d-1}+m-1}{m-1}$ matrix whose $\alpha i$-th element is $\big{(}A_{st}^d(r)\big{)}_{\a i}$.
We consider the block-matrix $A^d(r)= (A_{st}^d(r)),\,1\leq s,t\leq u_d$. 

Fix a $r, \,1\leq r\leq n_l -n_d$. To each $\mu= (\mu_1,\ldots,\mu_m)$ such that $|\mu| = n_l - n_d-r$, associate a $1\times\binom{n_l - n_d+m-1}{m-1}$ tuple of polynomials $X_{\mu r}^{ds}$ defined as follows:
$$
X_{\mu r}^{ds}(\beta) = \begin{cases}\mu !\binom{\beta}{\beta-\mu}\frac{\partial^{|\beta-\mu|}p_s^d}{\partial z^{\beta-\mu}} &\mbox{if}\,\, \beta\geq \mu\\ 0 &\mbox{otherwise,}\end{cases}
$$
where $\beta= (\beta_1,\ldots,\beta_m)$ with $|\beta|=n_l -n_d$. Let $X_{\mu r}^{d} = (X_{\mu r}^{d1},\ldots,X_{\mu r}^{d(n_l - n_d)})$. Using same argument as in \eqref{pd} and \eqref{pd1}, we see that the matrix
$$
A^d(r) = \sum_{|\mu|=n-r}\frac{1}{\mu!}(X_{\mu r}^{d}\cdot (X_{\mu r}^{d})^t)
$$
is non-negative definite when $r\geq 0$ and $A^d(0)$ is positive definite. Thus $A^d=\sum_{r=0}^{n_l -n_d}A^d(r)$ is positive definite. Let 
$$
\g_{lj}^d = ((\g_{lj}^{d1}(i))_{|i|=n_l-n_d},\ldots,(\g_{lj}^{d(n_l -n_d)}(i))_{|i|=n_l-n_d})^{tr},
$$ 
where each $(\g_{lj}^{ds}(i))_{|i|=n_l-n_d}$ is a $\binom{n_l - n_{d}+m-1}{m-1}\times 1$ column vector. Define
$$
\Gamma^d_{lj} = ((-\langle p_1^d,\frac{\partial^{|\a|}p_j^l}{\partial z^{\a}}\rangle_0)_{|\a|=n_l-n_d},\ldots,(-\langle p_{u_d}^d,\frac{\partial^{|\a|}p_j^l}{\partial z^{\a}}\rangle_0)_{|\a|=n_l-n_d}).
$$
The equation \eqref{m1} is then takes the form $\ov{A^d\g^d_{lj}} = \Gamma^d_{lj}$, which admits a solution (as $A^d$ is invertible) for each $d,l$ and $j$. Thus we have proved the following theorem.

\begin{thm}\label{hi}
Let $\mathcal I\subset\C[\underline{z}]$ be  a homogeneous ideal and  $\{p_1,\ldots,p_v\}$  be a minimal set of generators for $\mathcal I$ consisting of homogeneous polynomials. Let $K$ be the reproducing kernel corresponding the Hilbert module $[\mathcal I]$, which is assumed to be in $\mathfrak B_1(\Omega^*)$. Then there exists a set of generators $q_1,...,q_v$ for the ideal $\mathcal I$ such that the set $\{q_i(\bar
D)K(\cdot,w)|_{w=0}:\,1\leq i \leq v\}$ is a basis for $ \displaystyle\cap_{j=1}^n \ker M_j^*$.
\end{thm}

We remark that the new set of generators $q_1, \ldots , q_v$ for $\mathcal I$ is more or less ``\emph {canonical}''! It is uniquely determined modulo a linear transformation as shown below.

Let $\mathcal I\subset\C[\underline{z}]$ be an ideal. Suppose there are two sets of homogeneous polynomials $\{p_1,\ldots,p_v\}$ and $\{\tilde p_1,\ldots,\tilde p_v\}$ both of which are minimal set of generators for $\mathcal I$.  Theorem \ref{hi} guarantees the existence of a new set of generators $\{q_1,\ldots,q_v\}$ and $\{\tilde q_1,\ldots,\tilde q_v\}$ corresponding to each of these generating sets with additional properties which ensures that the equality
$$
[\tilde q_i^*]= \sum_{j=1}^v\alpha_{ij}[q_j^*],\,1\leq i\leq v
$$
holds in $\tilde{\mathbb V}_0(\mathcal I)/\mathbb V_0(\mathcal I)$ for some choice of complex constants $\alpha_{ij}$, $1\leq i,j\leq v$. Therefore $\tilde q_i^* - \sum_{i=1}^v\bar\alpha_{ij}q_j^*\in {\mathbb V}_{0}(\mathcal I)$. Since $\tilde q_i - \sum_{i=1}^v\alpha_{ij}q_j$ is in $\mathcal I,$ we have
$$
0 = \big ( (\tilde q_i^* - \sum_{i=1}^v\bar\alpha_{ij}q_j^* )(D) \big )\big (\tilde q_i - \sum_{i=1}^v\alpha_{ij}q_j\big ) = \parallel \tilde q_i - \sum_{i=1}^v\alpha_{ij}q_j\parallel^2_0,\,\, 1\leq i\leq v,
$$ 
and hence $\tilde q_i= \sum_{i=1}^v\alpha_{ij}q_j$, $1\leq i\leq v$. We have therefore proved the following.
\begin{prop}\label{vn}
Let $\mathcal I\subset \C[\underline{z}]$ be a homogeneous ideal.  If $\{q_1, \ldots, q_v\}$ is a minimal set of generators for $\mathcal I$ with the property that $\{[q_i^*]:\,1\leq i \leq v\}$ is a basis for $\tilde{\mathbb V}_0(\mathcal I)/\mathbb V_0(\mathcal I)$,  then $q_1, \ldots, q_v$ is unique up to a linear transformation.
\end{prop}


We end this section with the explicit calculation of the joint kernel for a class of submodules of the Hardy module which illustrate the methods of Proposition \ref{gen}.

\begin{ex}\label{ke} Let $p_1, p_2$ be the minimal set of generators for an ideal $\mathcal I \subseteq \mathbb C[z_1,z_2]$.  Assume that $p_1,p_2$ are homogeneous, $\deg p_2 = \deg p_1 +1$ and $V(\mathcal I) =\{0\}$. As in Proposition \ref{gen}, set $q_1=p_1$ and $q_2 = p_2 + (\gamma_{10}z_1 + \gamma_{01} z_2) p_1$ subject to the equations
 \begin{eqnarray} \label{Meq}\left(\begin{array}{ll}
\parallel\partial_1 p_1\parallel_0^2+\parallel p_1\parallel^2_0 & \langle \partial_2 p_1,\partial_1 p_1\rangle_0 \\
\langle \partial_1 p_1,\partial_2 p_1\rangle_0 & \parallel\partial_2 p_1\parallel_0^2+\parallel p_1\parallel^2_0   \end{array}\right) \left(\begin{array}{l}
 \gamma_{10}\\
\gamma_{01}
  \end{array}\right)
= - \left(\begin{array}{l}
 \langle p_1, \partial_1
p_2\rangle_0\\ \langle p_1,
\partial_2 p_2\rangle_0
  \end{array}\right)
\end{eqnarray}  
In this special case, the invertibility of the coefficient matrix follows from the positivity (Cauchy -  Schwarz inequality) of its  determinant  
\beqa \lefteqn{\parallel p_1\parallel^4_0+
\parallel\partial_1 p_1\parallel_0^2\parallel p_1\parallel^2_0+\parallel\partial_2 p_1\parallel_0^2\parallel p_1\parallel^2_0}\\ && \phantom{AAAAAAAAAA}+ (\parallel \partial_1 p_1\parallel_0^2 \parallel \partial_2 p_1\parallel_0^2 - |\langle \partial_1 p_1,\partial_2 p_1\rangle_0|^2).
\eeqa

Specifically, if the ideal  $\mathcal I\subset \C[z_1,z_2]$ is generated by $z_1+z_2$ and $z_2^2$. We have $V(\mathcal I)=\{0\}$. The reproducing kernel $K$ for $[\mathcal I]\subseteq H^2(\mathbb D^2)$ is 
\begin{eqnarray*}
K_{[\mathcal I]}(z,w)&=&\frac{1}{(1-z_1\bar w_1)(1-z_2\bar
w_2)}-\frac{(z_1-z_2)(\bar w_1 - \bar w_2)}{ 2} - 1\\
&=&\frac{(z_1 + z_2)(\bar w_1 + \bar w_2)}{ 2}+\d_{i+j\geq2}^\infty
z_1^iz_2^j\bar w_1^i\bar w_2^j.
\end{eqnarray*} 
The vector $\bar\partial_2^2K_{[\mathcal I]}(z,w)|_0= 2z_2^2$ is not in the joint kernel of $P_{[\mathcal I]}(M_1^*,M_2^*)|_{[\mathcal I]}$ since 
$M_2^*(z_2^2)= z_2$ and $P_{[\mathcal I]} z_2 = (z_1 + z_2)/2 \neq 0$. However, from the equation \eqref{Meq},  
we have $q_1=z_1+z_2$ and $q_2 = (z_1 - z_2)^2$,  we see that $q_1, q_2$ generate the ideal $\mathcal I$ and
$\{(\bar\partial_1+\bar\partial_2)K(\cdot,w)|_0,(\bar\partial_1-\bar\partial_2)^2K(\cdot,w)|_0\}$ forms a basis of the joint kernel. 

\end{ex}

\noindent\textbf{Remark on Example \ref{ke}.}
Let  $\wt{\mathcal I}$ be the ideal generated by $z_1$ and $z_2^2$. Since $z_1$ is not a linear combination of $q_1$ and $q_2$, it follows (Proposition \ref{vn}) that $\mathcal I\neq \wt{\mathcal I}$.  In fact Proposition \ref{vn} gives an effective tool to determine when a homogeneous ideal is monoidal. Let $\{q_1, \ldots, q_v\}$ be a canonical set of generators for $\mathcal I$. Let $\Lambda$ be the collection of monomials in the expressions of $\{q_1, \ldots, q_v\}$. If the number of algebraically independent monomials in $\Lambda$ is $v$, then $\mathcal I$ is monoidal.

\begin{rem}\label{ex}
If the generators of the ideal are not homogeneous then the conclusion of the theorem \ref{hi} is not valid. Take the ideal  $\mathcal I\subset \C[z_1,z_2]$ generated by $z_1(1+z_1),z_1(1-z_2),z_2^2$ which is also minimal for $\mathcal I$. We have $V(\mathcal I)=\{0\}$. We note that the stalk $\mathcal S^{\mathcal M}_0$ at $0$ is generated by $z_1$ and $z_2^2$. Similar calculations, as above, shows that $\{\bar\partial _1K(\cdot,w)|_0,\bar\partial_2^2K(\cdot,w)|_0\}$ is a basis
of $ \displaystyle\cap_{j=1}^2 \ker M_j^*$. But $z_1$ and $z_2^2$ can not be a set of generators for $\mathcal I\subset\C[z_1,z_2]$ which has rank $3$. On the other hand, let $\mathcal I$ be the ideal generated by $z_1+z_2+z_1^2, z_2^3-z_1^2$ which is minimal and $V(\mathcal I)=\{0\}$.  In this case $\{(\bar\partial_1+\bar\partial_2)K(\cdot,w)|_0,(\bar\partial_1-\bar\partial_2)^2K(\cdot,w)|_0\}$ is a basis
of $\cap_{j=1,2} \ker M_j^*$. But $z_1+z_2$ and $(z_1-z_2)^2$ is not a  generating set for the stalk at $0$.
\end{rem}

\section{Resolution of Singularities} \label{res}
We will use the familiar technique of `resolution of singularities'  and construct the blow-up space of $\Omega$ along an ideal $\mathcal I$, which we will denote by $\hat \Omega$. There is a map $\pi:\hat\Omega\ra\Omega$ which is biholomorphic on $\hat\Omega\setminus \pi^{-1}(V(\mathcal I))$. However, in general, $\hat\Omega$ need not even be a complex manifold. Abstractly, the inverse image sheaf of ${\mathcal S}^{\mathcal M}$ under $\pi$ is locally principal and therefore corresponds to a line bundle on $\hat\Omega$. Here, we explicitly construct a holomorphic line bundle, via the monoidal transformation, on $\pi^{-1}(w_0),\,w_0\in V(\mathcal I),$ and show that the equivalence class of these Hermitian holomorphic vector bundles are invariants for the Hilbert module $\mathcal M$.

In the paper \cite{dmv}, submodules of functions vanishing at the origin of $H^{(\lambda, \mu)}(\mathbb D^2)$ were studied using the blow-up $\mathbb D^2 \setminus (0,0) \cup \mathbb P^1$ of the bi-disc.  This is also known as the quadratic transform.  However, this technique yields useful information only if the generators of the submodule are homogeneous polynomials of same degree. The monoidal transform, as we will see below, has wider applicability.

For any two Hilbert module  $\mathcal M_1$ and $\mathcal M_2$ in the class $\mathcal B_1(\Omega)$ and $L: \mathcal M_1\ra \mathcal M_2$ a module map between them, let $\mathcal S^L: \mathcal S^{\mathcal M_1}(V)\ra \mathcal S^{\mathcal M_2}(V)$ be the map defined by
$$
\mathcal S^L{\sum_{i=1}^nf_i|_Vg_i}:=\sum_{i=1}^nLf_i|_Vg_i, \mbox{~for~} f_i\in\mathcal M_1, g_i\in\mathcal O(V),\,n\in\mathbb N.
$$
The map $\mathcal S^L$ is well defined: if  $\sum_{i=1}^nf_i|_Vg_i=\sum_{i=1}^n\tilde f_i|_V\tilde g_i$, then $\sum_{i=1}^nLf_i|_Vg_i=\sum_{i=1}^nL\tilde f_i|_V\tilde g_i$. Suppose $\mathcal M_1$ is isomorphic to $\mathcal M_2$ via the unitary module map $L$. Now, it is easy to verify that $(\mathcal S^{L})^{-1}=\mathcal S^{L^*}$.  It then follows that $\mathcal S^{M_1}$ is isomorphic, as sheaves of modules over $\mathcal O_{\Omega}$, to $\mathcal S^{M_2}$ via the map  $\mathcal S^{L}$.

Let $K_i$ be the reproducing kernel corresponding to $\mathcal M_i,\, i=1,2$. 
We assume that the dimension of the zero sets $X_i=Z(\mathcal M_i)$ of the modules $\mathcal M_i,\,i=1,2,$ is less or equal to $m-2$. Recall that the stalk $\mathcal S^{\mathcal M_i}_{w}$ is $\mathcal O_{w}$ for $w\in\Omega\setminus X_1,\,i=1,2$. Let $X=X_1\cup X_2$. From \cite[Lemma 1.3]{bmp} and \cite[Theorem 3.7]{cs}, it follows that there exists a non-vanishing holomorphic function $\phi:\Omega \setminus X \ra \C$
such that $LK_1(\cdot,w)= \bar\phi(w)K_2(\cdot,w),~ L^*f=\phi f$ and $K_1(z,w) = \phi(z)K_2(z,w)\bar\phi(w)$.
The function $\psi=1/\phi$ on $\Omega\setminus X$ (induced by the inverse of $L$, that is, $L^*)$ is holomorphic. Since $\dim X\leq m-2$, by Hartog's theorem (cf. \cite[Page 198]{loj})
there is a unique extension of $\phi$ to $\Omega$ such that $\phi$ is non-vanishing on $\Omega$ ($\psi$ have an extension to $\Omega$ and $\phi\psi=1$ on the open set $\Omega\setminus X$). Thus $X_1=X_2$. For
$w_0\in X$,  the stalks are not just isomorphic but equal:
\begin{eqnarray*}\mathcal S^{\mathcal M_1}_{w_0}&=& \{\d_{i=1}^nh_ig_i : g_i\in\mathcal M_1, h_i\in {_m\mathcal O_{w_0}},
1\leq i\leq n, n\in\N\}\\ &=& \{\d_{i=1}^nh_i\phi f_i : f_i\in\mathcal M_2, h_i\in {_m\mathcal O_{w_0}}, 1\leq
i\leq n, n\in\N\}\\ &=& \{\d_{i=1}^n\tilde h_if_i : f_i\in\mathcal M_2, \tilde h_i\in {_m\mathcal O_{w_0}},
1\leq i\leq n, n\in\N\}=\mathcal S^{\mathcal M_2}_{w_0}.\end{eqnarray*}

The following theorem is modeled after the well known rigidity theorem which is obtained by taking $\mathcal M = \wt{\mathcal M}$. The proof below is different from the ones in \cite{cg} or \cite{dpsy} and uses the techniques developed in this paper and in \cite{bmp}. We note the conditions in \cite[Theorem 3.6]{dpsy} are same as the following theorem, as dimension of the algebraic variety $V(\mathcal I)$ for some ideal $\mathcal I\subset \C[\underline z]$ is same as the holomorphic dimension by \cite[Theorem 5.7.1]{tay}.

\begin{thm}
Let $\mathcal M$ and $\wt{\mathcal
M}$ be two Hilbert
modules in $\mathfrak B_1(\Omega^*)$ consisting of holomorphic functions on a bounded domain $\Omega\subset \C^m$. Assume that the dimension of the zero set of these modules is at most $m-2$. Suppose there exists polynomial ideals $\mathcal I$ and $\wt{\mathcal I}$ such that $\mathcal M = [\mathcal I]_{\mathcal M}$ and $\wt{\mathcal M} = [\wt{\mathcal I}]_{\wt{\mathcal M}}$. Assume that every algebraic component of $V(\mathcal I)$ and $V(\wt{\mathcal I})$ intersects $\Omega$. Then $\mathcal M$ and $\wt{\mathcal
M}$ are equivalent if and only if $\mathcal I = \wt{\mathcal I}$.
\end{thm}
\begin{proof}
For $w_0\in\Omega$, we have $\mathbb V_{w_0}(\mathcal I)=\mathbb V_{w_0}(\mathcal S^{\mathcal M}_{w_0})$ from \cite[Lemma 3.2 and 3.3]{bmp},
and $\mathcal S^{\mathcal M}_{w_0}=\mathcal S^{\wt{\mathcal M}}_{w_0}$. Therefore  $\mathbb V_{w_0}(\mathcal I) = \mathbb V_{w_0}(\wt{\mathcal I})$. In other words, setting $\mathcal I^e_{w_0} = \{p\in\C[\underline z]: q(D)p|_{w_0}=0 \mbox{~for~all~}q\in\mathbb V_{w_0}(\mathcal I)( = \mathbb V_{w_0}\}$, as in \cite{cg}, we see that $\mathcal I^e_{w_0} = {\wt{\mathcal I}}^e_{w_0}$ for all $w_0\in\Omega$. The proof is now complete since $
\mathcal I = \cap_{w_0\in\Omega}\mathcal I^e_{w_0}$ (cf. \cite[Corollary 2.1.2]{cg}).
\end{proof}
\begin{ex}
For $j=1,2$, let $\mathcal I_j\subset \C[z_1,\ldots , z_m]$, $m > 2$, be the ideals generated by $z_1^n$ and $z_1^{k_j}z_2^{n-k_j}$.  Let $[\mathcal I_j]$ be the submodule in the Hardy module $H^2(\mathbb D^m)$. Now, from the Theorem proved above, it follows that $[\mathcal I_1]$ is equivalent to $[\mathcal I_2]$ if and only if $\mathcal I_1 = \mathcal I_2$.  We conclude, using Proposition \ref{vn}, that these two ideals are same only if $k_1= k_2$.
\end{ex}

\subsection{\sf \sf The Monoidal Transformation}
Let $\mathcal M$ be a Hilbert module in $\mathfrak B_1(\Omega^*)$, which is the closure, in $\mathcal M$, of some polynomial ideal $\mathcal I$. Let $K$ denote the corresponding reproducing kernel.  Let $w_0\in Z(\mathcal M)$. Set 
$$
t=\dim\mathcal S^\mathcal M_{w_0}/\mathfrak
m_{w_0} \mathcal S^\mathcal M_{w_00} = \dim\cap_{j=1}^m\ker (M_j - w_{0j}) ^*= \dim \tilde{\mathbb V}_{w_0}(\mathcal I)/\mathbb V_{w_0}(\mathcal I).
$$ 
By the decomposition theorem \cite[Theorem 1.5]{bmp}, there exists a minimal set of generators $g_1,\cdots,g_t$ of $\mathcal S^{\mathcal M_1}_{0}$ and  a $r>0$ such that
\beq\label{dec}
K(\cdot, w)=\sum_{i=1}^t\ov{g_j(w)}K^{(j)}(\cdot,w) \mbox{~for~all~}w\in\Delta(w_0; r)
\eeq
for some choice of anti-holomorphic functions $K^{(1)}, \ldots , K^{(t)}:\Delta(w_0; r) \to\mathcal M$.

Assume  that $Z:= Z(g_1,\ldots, g_t)\cap \Omega$ be a singularity free analytic subset of $\C^m$ of codimension $t$. We point out that $Z$ depends on $\mathcal M$ as well as $w_0$. Define
$$
\widehat{\Delta}(w_0;r):= \{(w,\pi(u))\in \Delta(w_0;r)\times {\mathbb P}^{t-1}: u_ig_j(w) - u_jg_i(w)=0 ,\,1\leq i,j\leq t\}.
$$
Here the map $\pi:\C^t\setminus\{\underline 0\}\ra {\mathbb P}^{t-1}$ is given by $\pi(u) = (u_1:\ldots:u_t)$, the corresponding projective coordinate. The space $\widehat{\Delta}(w_0;r)$ is the monoidal transformation with center $Z$ (\cite[page 241]{fg}). Consider the map  $p:= \mbox{pr}_1: \widehat{\Delta}(w_0;r)\ra \Delta(w_0;r)$ given by $(w,\pi(z))\mapsto w$. For $w\in Z,$ we have $p^{-1}(w) = \{w\}\times {\mathbb P}^{t-1}$. This map is holomorphic and proper. Actually $p:\widehat{\Delta}(w_0;r)\setminus p^{-1}(Z)\ra \Delta(w_0;r)\setminus Z$ is biholomorphic with $p^{-1}: w\mapsto (w,  (g_1(w):\ldots:g_t(w)))$. The set $E(\mathcal M):= p^{-1}(Z)$ which is $Z\times {\mathbb P}^{t-1}$, is called the exceptional set.

We describe a natural line bundle on the blow-up space $\widehat{\Delta}(w_0;r)$. Consider the open set $U_1= (\Delta(w_0;r)\times\{u_1\neq 0\})\cap \widehat{\Delta}(w_0;r)$. Let $\frac{u_j}{u_1} = \theta^1_j,\,2\leq j\leq t$. On this chart $g_j(w) = \theta^1_jg_j(w)$. From the decomposition given in the equation \eqref{dec}, we have
$$
K(\cdot, w)=\ov{g_1(w)}\{K^{(1)}(\cdot,w) + \sum_{j=2}^t\bar\theta^1_jK^{(j)}(\cdot,w)\}.
$$
This decomposition then yields a section on the chart $U_1$, of the line bundle on the blow-up space $\widehat{\Delta}(w_0;r)$:
$$
s_1(w,\theta) = K^{(1)}(\cdot,w) + \sum_{j=2}^t\bar\theta^1_jK^{(j)}(\cdot,w).
$$
The vectors $ K^{(j)}(\cdot,w)$ are not uniquely determined. However, there exists a canonical choice of these vectors starting from a basis, $\{v_1,\ldots,v_t\}$, of the joint kernel $\cap_{i=1}^n\ker (M_j - w_j)^*$:
$$
K(\cdot,w)= \sum_{j=1}^t \ov{g_j(w)}P(\bar w, \bar w_0)v_j,\, w\in\Delta(w_0;r)
$$
for some $r>0$ and generators $g_1,\ldots, g_t$ of the stalk $\mathcal S^{\mathcal M}_{w_0}$. Thus we obtain the canonical choice $K^{(j)}(\cdot,w) = P(\bar w, \bar w_0)v_j,\, 1\leq j\leq t$ (cf. \cite[Section 6]{bmp}). Let $\mathcal L(\mathcal M)$ be the line bundle on the blow-up space $\widehat{\Delta}(w_0;r)$ determined by the section  $(w,\theta)\mapsto s_1(w,\theta)$, where
$$
s_1(w,\theta) = P(\bar w, \bar w_0)v_1+\sum_{j=2}^t \bar\theta^1_j P(\bar w, \bar w_0)v_j ,\, (w,\theta)\in U_1.
$$
Let $\wt{\mathcal M}$ be a second Hilbert module in $\mathfrak B_1(\Omega^*)$, which is the closure of the polynomial ideal ${\mathcal I}$ with respect to another inner product. Assume that  $\wt{\mathcal M}$ is equivalent to $\mathcal M$  via a unitary module map $L$. In the proof of Theorem 1.10 in \cite{bmp}, we have shown that $LP(\bar w, \bar w_0) ={\wt P}(\bar w, \bar w_0)L$. Thus
$$
\ov{\phi(w)}\wt K(\cdot,w) = LK(\cdot,w)= \sum_{j=1}^t \ov{g_j(w)}LP(\bar w, \bar w_0)v_j = \sum_{j=1}^t \ov{g_j(w)}\wt{P}(\bar w, \bar w_0)Lv_j.
$$
Therefore $s_1(w,\theta) = \frac{1}{\ov{\phi(w)}}(\wt{P}(\bar w, \bar w_0)Lv_1+\sum_{j=2}^t \bar\theta^1_j \wt{P}(\bar w, \bar w_0)Lv_j)$ and 
$$
Ls_1(w,\theta) = \ov{\phi(w)}\tilde s_1(w,\theta).
$$ 
Hence the line bundles $\mathcal L(\mathcal M)$ and $\mathcal L(\wt{\mathcal M})$ are equivalent as Hermitian holomorphic line bundle on $\widehat{\Delta}(w_0;r)^*=\{(\bar w,\pi(\bar u)): (w,\pi(u))\in\widehat{\Delta}(w_0;r)\}$. Since $K^{(j)}(\cdot,w),1\leq j\leq t$ are linearly independent \cite[Theorem 1.5]{bmp}, it follows that $Z(\mathcal M)\cap\Delta(w_0;r) = Z$. Thus  if $w\in\Delta(w_0;r)\setminus Z$, then $g_i(w)\neq 0$ for some $i,\,1\leq i\leq t$. Hence $s_i(w,\theta)= \frac{k(\cdot,w)}{\ov {g_i(w)}}$ on $(\Delta(w_0;r)\times\{u_i\neq 0\})\cap \widehat{\Delta}(w_0;r)$. Therefore the restriction of the bundle $\mathcal L(\mathcal M)$ to $\widehat{\Delta}(w_0;r)\setminus p^{-1}(Z)$ is the pull back of the Cowen-Douglas bundle for $\mathcal M$ on $\Delta(w_0;r)\setminus Z$, via the biholomorphic map $\pi$ on $\widehat{\Delta}(w_0;r)\setminus p^{-1}(Z)$. we have therefore proved the following Theorem.

\begin{thm}\label{bu}
Let $\mathcal M$ and $\wt{\mathcal
M}$ be two Hilbert
modules in $\mathfrak B_1(\Omega)$ consisting of holomorphic functions on a bounded domain $\Omega\subset \C^m$. Assume that the dimension of the zero set of these modules is at most $m-2$. Suppose there exists a polynomial ideal $\mathcal I$ such that $\mathcal M$ and $\wt{\mathcal M}$ are the completions of $\mathcal I$ with respect to different inner product. Then $\mathcal M$ and $\wt{\mathcal M}$ are equivalent if and only if the line bundles $\mathcal L(\mathcal M)$ and $\mathcal L(\wt{\mathcal M})$ are equivalent as Hermitian holomorphic line bundle on $\widehat{\Delta}(w_0;r)^*$.
\end{thm}

Although in general, $Z$ need not be a complex manifold, The restriction of $s_1$ to $p^{-1}(w_0)$ for $w_0\in Z$ determines a holomorphic line bundle on $p^{-1}(w_0)^*:=\{(w_0,\pi(\bar u)): (\bar w_0,\pi(u))\in p^{-1}(w_0)\}$, which we denote by $\mathcal  L_0(\mathcal M)$. Thus $s_1 = s_1(w,\theta)|_{\{w_0\}\times \{u_i\neq 0\}}$ is given by the formula
$$
s_1(\theta) = K^{(1)}(\cdot,w_0) + \sum_{j=2}^t\bar\theta^1_jK^{(j)}(\cdot,w_0).
$$
Since the vectors $K^{(j)}(\cdot,w_0),1\leq j\leq t$ are uniquely determined by the generators $g_1,\ldots,g_t$, $s_1$ is well defined.

\begin{thm} \label{proj}
Let $\mathcal M$ and $\wt{\mathcal
M}$ be two Hilbert
modules in $\mathfrak B_1(\Omega)$ consisting of holomorphic functions on a bounded domain $\Omega\subset \C^m$. Assume that the dimension of the zero set of these modules is at most $\leq m-2$. Suppose there exists a polynomial ideal $\mathcal I$  such that $\mathcal M$ and $\wt{\mathcal M}$ are the completions of $\mathcal I$ with respect to different inner product.  If the
modules $\mathcal M$ and $\wt{\mathcal M}$ are equivalent, then the
corresponding bundles $\mathcal L_0(\mathcal M)$ and $\mathcal L_0(\wt{\mathcal M})$ they determine on the projective space $p^{-1}(w_0)^*$ for $w_0\in Z$, are equivalent as Hermitian holomorphic line bundle.
\end{thm}
\begin{proof} Let $L:\mathcal M\ra \wt{\mathcal M}$ be the unitary module map and $K$ and $\wt K$ be the reproducing kernels corresponding to $\mathcal M$ and $\wt{\mathcal M}$ respectively. The existence of a holomorphic function $\phi$ on $\Omega\setminus Z(\mathcal M)$ such that $LK(\cdot,w)= \ov{\phi(w)}\wt K(\cdot,w),~ L^*f=\phi f$ and $K(z,w) = \phi(z)\wt K(z,w)\ov{\phi(w)}$ follows from  \cite[Lemma 1.3]{bmp} and \cite[Theorem 3.7]{cs}. As we have pointed earlier,  $\phi$ extends to a non-vanishing holomorphic function on $\Omega$.

Since $\mathcal M$ is in $\mathfrak B_1(\Omega^*)$, it admits a decomposition as given in equation \eqref{dec}, with respect the generators $\tilde g_1,\ldots,\tilde g_t$ of $\mathcal S^{\wt{\mathcal M}}_{w_0}$. However, we may assume that $\tilde g_i = g_i$ for $1\leq i\leq t$, because $\mathcal S^{\mathcal M}_{w_0}=\mathcal S^{\wt{\mathcal M}}_{w_0}$ for all $w_0\in\Omega$. Thus
\beqa
\wt K(\cdot, w)=\sum_{i=1}^t\ov{g_j(w)}\wt K^{(j)}(\cdot,w) \mbox{~for~all~}w\in\Delta(w_0; r)
\eeqa
For some $r>0$. By applying the unitary $L$ to equation \eqref{dec}, we get 
$$
\ov{\phi(w)}\wt K(\cdot, w) = LK(\cdot,w)~=~ \sum_{i=1}^t\ov{g_j(w)}LK^{(j)}(\cdot,w).
$$ 
Since $\phi$ does not vanish on $\Omega$, we may choose
$$
\wt K^{(j)}(\cdot,w) = \frac{LK^{(j)}(\cdot,w)}{\ov{\phi(w)}},\, 1\leq j\leq t,\, w\in \Delta(w_0; r).
$$
From part (iii) of the decomposition theorem (\cite[Theorem 1.5]{bmp}),  the vectors $\wt K^{(j)}(\cdot,w_0),\,1\leq j\leq t$ are uniquely determined by the generators $g_1,\ldots,g_t$. Therefore $\wt K^{(j)}(\cdot,w_0) = \frac{LK^{(j)}(\cdot,w_0)}{\ov{\phi(w_0)}}$. Now the decomposition for $\wt K$ yields a holomorphic section ${\tilde s}_1(\theta) = \wt K^{(1)}(\cdot,w_0) + \sum_{j=2}^t\theta^1_j\wt K^{(j)}(\cdot,w_0)$ for the holomorphic line bundle $\mathcal L_{0}(\wt{\mathcal M})$ on the projective space $p^{-1}(w_0)^*$. Therefore
\beqa
Ls_1(\theta) &=& LK^{(1)}(\cdot,w_0) + \sum_{j=2}^t\bar\theta^1_jLK^{(j)}(\cdot,w_0)\\ &=& \ov{\phi(w_0)}\{\wt K^{(1)}(\cdot,w_0) + \sum_{j=2}^t\bar\theta^1_j\wt K^{(j)}(\cdot,w_0)\}~=~ \ov{\phi(w_0)}{\tilde s}_1(\theta).
\eeqa
From the unitarity of $L$, it follows that 
\beq\label{sec}\parallel s_1(\theta)\parallel^2 = \parallel Ls_1(\theta)\parallel^2 = |\phi(w_0)|^2\parallel {\tilde s}_1(\theta)\parallel^2
\eeq and consequently the Hermitian holomorphic line bundles $\mathcal L_0(\mathcal M)$ and $\mathcal L_0(\wt{\mathcal M})$ on the projective space $p^{-1}(w_0)^*$ are equivalent.
\end{proof}

The existence of the polynomials $q_1,...,q_t$ such that $K^{(j)}(\cdot,w)|_{w=w_0}= q^*_j(\bar D)K(\cdot,w)|_{w=w_0},\, 1\leq j\leq t$, is guaranteed by Lemma \ref{nice}. The following Lemma shows that $$
\wt K^{(j)}(\cdot,w)|_{w=w_0}= q_j^*(\bar D)\wt K(\cdot,w)|_{w=w_0},\,1\leq j\leq t
$$
which makes it possible to calculate the section for the line bundles $\mathcal L_0(\mathcal M)$ and $\mathcal L_0(\wt{\mathcal M})$ without any explicit reference to the generators of the stalks at $w_0$.

\begin{lem}
Let $\mathcal I$ be a polynomial ideal with dim $V(\mathcal I)\leq m-2$ and $K$ be the reproducing kernel of $[\mathcal I]$ which is assumed to be in $\mathfrak B_1(\Omega^*)$. Let $q_1,...,q_t$ be the polynomials such that $K^{(j)}(\cdot,w)|_{w=w_0}= q^*_j(\bar D)K(\cdot,w)|_{w=w_0}$. Let $\wt K$ be a reproducing kernel of $[\mathcal I]$, completed with respect to another inner product. Then $\wt K^{(j)}(\cdot,w)|_{w=w_0}= q^*_j(\bar D)\wt K(\cdot,w)|_{w=w_0}$.
\end{lem}

\begin{proof}
For $f\in\mathcal M$ and $1\leq i\leq m$, we have $\langle f,\bar\partial_iL
K(\cdot,w)\rangle=\partial_i\langle f,L
K(\cdot,w)\rangle= \partial_i\langle L^*f,
K(\cdot,w)\rangle=\langle L^*f,\bar\partial_i
K(\cdot,w)\rangle = \langle f,L\bar\partial_i
K(\cdot,w)\rangle$, that is, $\bar\partial_iL
K(\cdot,w)=L\bar\partial_i K(\cdot,w)$. Thus
$$
p(\bar D)LK(\cdot,w) =  Lp(\bar D)K(\cdot,w) \mbox{~for~ any~} p\in\C[\underline z].
$$
From equation \eqref{gf}, it follows that
\beqa
LK^{(j)}(\cdot,w_0) &=& L\{q^*_j(\bar D)K(\cdot,w)|_{w=w_0} \}= \{Lq^*_j(\bar D)K(\cdot,w)\}|_{w=w_0}\\ &=& \{q^*_j(\bar D) LK(\cdot,w)\}|_{w=w_0}~=~ \{q^*_j(\bar D) \ov{\phi(w)}\wt K(\cdot,w)\}|_{w=w_0}\\ &=& [\sum_{\alpha}\bar a_{\alpha}\{q^*_j(\bar D) (\bar w - \bar w_0)^{\alpha}\wt K(\cdot,w)\}]|_{w=w_0}\\ &=& \sum_{\alpha}\bar a_{\alpha} \frac{\partial^\alpha q^*_j}{\partial z^{\alpha}}(\bar D)\wt K(\cdot,w)|_{w=w_0},
\eeqa
where $\phi(w) = \sum_{\alpha}a_{\alpha}(w - w_0)^{\alpha}$, the power series expansion of $\phi$ around $w_0$. Now for any $p\in\mathcal I$ we have
\beqa
\langle p, \frac{\partial^\alpha q^*_j}{\partial z^{\alpha}}(\bar D)\wt K(\cdot,w)|_{w=w_0}\rangle &=& \langle p, \frac{\partial^\alpha q^*_j}{\partial z^{\alpha}}(\bar D)\wt K(\cdot,w)\rangle|_{w=w_0}\\ &=& \frac{\partial^\alpha q_j}{\partial z^{\alpha}}(D)p(w)|_{w=w_0}.
\eeqa
Since Lemma \ref{nice} ensures that $\{[q_1],\ldots,[q_t]\}$ is a basis for $\tilde{\mathbb V}_{w_0}(\mathcal I)/\mathbb V_{w_0}(\mathcal I)$, it follows that
$$
\langle p, \frac{\partial^\alpha q^*_j}{\partial z^{\alpha}}(\bar D)\wt K(\cdot,w)|_{w=w_0}\rangle =0 \mbox{~for~ all~} p\in\mathcal I \mbox{~and~} \alpha>0.
$$
Therefore, we have $\frac{\partial^\alpha q^*_j}{\partial z^{\alpha}}(\bar D)\wt K(\cdot,w)|_{w=w_0} = 0$ for $\alpha>0$. Hence $LK^{(j)}(\cdot,w_0) = \bar a_0q^*_j(\bar D)\wt K(\cdot,w)|_{w=w_0} = \ov{\phi(w_0)}q^*_j(\bar D)\wt K(\cdot,w)|_{w=w_0}$and consequently $\wt K^{(j)}(\cdot,w)|_{w=w_0}= q^*_j(\bar D)\wt K(\cdot,w)|_{w=w_0},\, 1\leq j\leq t$. \end{proof}

\begin{rem}\label{m-1}
Let $\mathcal M$ Be a Hilbert module in $\mathfrak B_1(\Omega)$. Assume that $\mathcal M = [\mathcal I]_{\mathcal M}$ for some polynomial ideal $\mathcal I$ and the dimension of the zero set of $\mathcal M$ is $m-1$. Let the polynomials $p_1,\ldots,p_t$ be a minimal set of generators for $\mathcal M$. Let $q =$ g.c.d$\{p_1,\ldots,p_t\}$. Then the Beurling form (cf. \cite{cg}) of $\mathcal I$ is $q\mathcal J$, where $\mathcal J$ is generated by $\{p_1/q,\ldots, p_t/q\}$. From \cite[Corollary 3.1.12]{cg}, dim $V(\mathcal J)\leq m-2$ unless $\mathcal J=\C[\underline z]$. The reproducing kernels $K$ of $\mathcal M$ is of the form $K(z,w) = q(z)\chi(z,w)\ov{q(w)}$. Let $\mathcal M_1$ be the Hilbert module determined by the non-negative definite kernel $\chi$. The Hilbert module $\mathcal M$ is equivalent to $\mathcal M_1$. Now $\mathcal M_1 = [\mathcal J]$ and $Z(\mathcal M_1) = V(\mathcal  J)$. If $V(\mathcal  J) =\phi$, then the modules $\mathcal M_1$ belongs to Cowen-Douglas class of rank $1$. Otherwise, dim $V(\mathcal  J)\leq m-2$ and Theorem \ref{bu} determines its equivalence class.
\end{rem}

\section{Examples}

We illustrate, by means of some examples, the nature of the
invariants we obtain from the line bundle $\mathcal L_0$ that lives on the
projective space. From Theorem \ref{proj}, it follows that the
curvature of the line bundle $\mathcal L_0$ is an invariant for the submodule. An example was given in \cite{dmv} showing that the
curvature is is not a complete invariant. However the following lemma is useful for obtaining complete invariant in a large class of examples. 

\begin{lem} \label{16} Let $\mathcal H$ and $\wt{\mathcal H}$
are Hilbert modules in $\mathfrak B_1(\Omega)$ for some bounded domain $\Omega$ in $\C^m$. Suppose that $\mathcal H$ and $\wt{\mathcal H}$ are such that they are in the Cowen-Douglas class $B_1(\Omega\setminus X)$ where $\dim X\leq m-2$. Let $\mathcal M$ and $\wt{\mathcal M}$ be any submodules of $\mathcal H$ and $\wt{\mathcal H}$ respectively, such that
\begin{enumerate} 
\item[(i)]$\mathbb V_w(\mathcal M) = \mathbb V_w(\wt{\mathcal M})$ for all $w\in\Omega$ and
\item[(ii)] $\mathcal M = \cap_{w\in\Omega} \mathcal M^e_w$ and $\wt{\mathcal M} = \cap_{w\in\Omega} {\wt{\mathcal M}}^e_w$, where as before $\mathcal M^e_w := \{f\in\mathcal H: q(D)f|_{w}=0 \mbox{~for~all~} q\in\mathbb V_w(\mathcal M)\}$.
\end{enumerate}
 If $\mathcal H$ and $\wt{\mathcal H}$ are equivalent, then $\mathcal M$ and $\wt{\mathcal M}$ are equivalent.
\end{lem}

\begin{proof} Suppose $U:\mathcal H\ra\wt{\mathcal H}$ is a unitary module map. Then $U$ is is a multiplication operator induced  by some holomorphic function, say $\psi$, on $\Omega\setminus X$ (cf. \cite{cs}). This function $\psi$ extends non-vanishingly to all of $\Omega$ by Hartog's Theorem. Let $w_0\in\Omega$ and $q\in\mathbb V_{w_0}(\mathcal M) = \mathbb V_{w_0}(\wt{\mathcal M})$.
Also let $\psi(w)= \sum_{\alpha}a_{\alpha}(w - w_0)^{\alpha}$ be the power series expansion around $w_0$. For $f\in\mathcal M$, we have
\beqa
q(D)(Uf)|_{w=w_0} &=& q(D)(\psi f)|_{w=w_0} ~=~ q(D)\{\sum_{\alpha}a_{\alpha}(w - w_0)^{\alpha} f\}|_{w=w_0}\\ &=& \sum_{\alpha}a_{\alpha}q(D)\{(w - w_0)^{\alpha} f\}]|_{w=w_0} \\ &=&\{\sum_{k\leq\alpha} \binom{\alpha}{k} (w -w_0)^{\alpha-k}\frac{\partial^kq}{\partial z^k}(D)(f)\}_{w=w_0}\\ &=& 0
\eeqa 
since $\frac{\partial^kq}{\partial z^k}\in\mathbb V_{w_0}(\mathcal M)$ for any multi index $k$ whenever $q\in\mathbb V_{w_0}(\mathcal M)$. Therefore  it follows that $Uf\in\wt{\mathcal M}$. A similar arguments shows that $U^*\wt{\mathcal M}\subseteq \mathcal M$. The result follows from unitarity of $U$.
\end{proof}

\subsection{\sf  The $(\alpha, \beta, \theta)$ examples: Weighted Bergman Modules in the unit ball}
Let $\mathbb B^2 = \{z = (z_1,z_2)\in\C^2: |z_1|^2+|z_2|^2<1\}$ be the unit ball in $\C^2$. Let $L^2_{\alpha,\beta,\theta}(\mathbb B^2)$ be the Hilbert space of all (equivalence classes of) Borel measurable functions on $\mathbb B^2$ satisfying 
$$
\parallel f\parallel^2_{\alpha,\beta,\theta} = \int_{\mathbb B^2} |f(z)|^2d\mu(z_1,z_2)<+\infty,
$$
where the measure is 
$$
d\mu(z_1,z_2) = (\alpha+\beta+\theta+2)|z_2|^{2\theta}(1 - |z_1|^2 - |z_2|^2)^{\alpha}(1 - |z_2|^2)^{\beta}dA(z_1,z_2)
$$
for $(z_1,z_2)\in\mathbb B^2$, $-1<\alpha,\beta,\theta<+\infty$ and $dA(z_1,z_2) = dA(z_1)dA(z_2)$. Here $dA$ denote the normalized area measure in the plane, that is $dA(z) = \frac{1}{\pi} dxdy$ for $z=x+iy$. The weighted Bergman space $\mathcal A^2_{\alpha,\beta,\theta}(\mathbb B^2)$ is the subspace of $L^2_{\alpha,\beta,\theta}(\mathbb B^2)$ consisting of the holomorphic functions on $\mathbb B^2$. The Hilbert space $\mathcal A^2_{\alpha,\beta,\theta}(\mathbb B^2)$ is non-trivial if we assume that the parameters $\alpha,\beta,\theta$ satisfy the additional condition:
$$
\alpha+\beta+\theta +2>0.
$$
The reproducing kernel $K_{\alpha,\beta,\theta}$ of $\mathcal A^2_{\alpha,\beta,\theta}(\mathbb B^2)$ is given by
\beqa
K_{\alpha,\beta,\theta}(z,w) &=& \frac{1}{\alpha+\beta+\theta +2}\frac{1}{(1- z_1\bar w_1)^{\alpha+\beta+\theta+3}}\\ &&\times\bigg\{\sum_{k=0}^{+\infty}\frac{(\alpha+\beta+\theta+k+2)(\alpha+\theta+2)_k}{(\theta+1)_k}\bigg(\frac{z_2\bar w_2}{1- z_1\bar w_1}\bigg)^k\bigg\}, 
\eeqa
where $z=(z_1,z_2), w=(w_1,w_2)\in\mathbb B^2$ and $(a)_k= a(a+1)\ldots(a+k-1)$ is the Pochhammer symbol. This kernel differs from the kernel $P_{\alpha,\beta,\theta}$ given in \cite{hss} only by a multiplicative constant. The reader may consult \cite{hss} for a detailed discussion of these Hilbert modules. 
 
Let $\mathcal I_P$ be an ideal in $\C[z_1,z_2]$ such that $V(\mathcal I_P) = \{P\}\subset\mathbb B^2$. We have
\beqa
\dim \ker D_{(M - w)^*}=\left\{
\begin{array}{ll}
1 & \hbox{for $w\in{{\mathbb B}^2\setminus\{P\}}$;} \\
\dim \mathcal I_P/\mathfrak m_P\mathcal I_P\,(\,>1\,) & \hbox{for $w=P$.}
\end{array}
\right.
\eeqa
Hence $[\mathcal I_P]_{{\mathcal A}^2_{\alpha,\beta,\theta}(\mathbb B^2)}$ (the completion of $\mathcal I_P$ in ${{\mathcal A}^2_{\alpha,\beta,\theta}(\mathbb B^2)}$) is not equivalent to $[\mathcal I_{P'}]_{{\mathcal A}^2_{\alpha',\beta',\theta'}(\mathbb B^2)}$ (the completion of $\mathcal I_P'$ in ${{\mathcal A}^2_{\alpha',\beta',\theta'}(\mathbb B^2)}$) if $P\neq P'$. Now let us determine
when two modules in the set
$$
\{[\mathcal I_P]_{{\mathcal A}^2_{\alpha,\beta,\theta}(\mathbb B^2)}: -1<\alpha,\beta,\theta<+\infty \mbox{~and~}\alpha+\beta+\theta +2>0\}.
$$ 
are equivalent. In the following
proposition, without loss of generality, we have assumed $P=0$. 

\begin{prop}
Suppose $\mathcal I$ is an ideal in $\C[z_1,z_2]$ with $V(\mathcal I) = \{0\}$. Then the Hilbert modules $[\mathcal I]_{{\mathcal A}^2_{\alpha,\beta,\theta}(\mathbb B^2)}$ and $[\mathcal I]_{{\mathcal A}^2_{\alpha',\beta',\theta'}(\mathbb B^2)}$ are unitarily equivalent if and only
if $\alpha=\alpha', \beta = \beta'$ and $\theta=\theta'$.
\end{prop}

\begin{proof} From the Hilbert Nullstellensatz, it follows that there exist an natural number $N$ such that $\mathfrak m_0^N\subset \mathcal I$. Let $\mathcal I_{m,n}$ be the polynomial ideal generated by $z_1^m$ and $z_2^n$. Combining \eqref{obs1} with Lemma \ref{16} we see, in particular, that the submodules $[\mathcal I_{m,n}]_{{\mathcal A}^2_{\alpha,\beta,\theta}(\mathbb B^2)}$  and $[\mathcal I_{m,n}]_ {{\mathcal A}^2_{\alpha',\beta',\theta'}(\mathbb B^2)}$  are unitarily equivalent for $m,n\geq N$. Let $K_{m,n}$  be the reproducing kernel for $[\mathcal I_{m,n}]_{{\mathcal A}^2_{\alpha,\beta,\theta}(\mathbb B^2)}$.  We write $K_{\alpha,\beta,\theta}(z,w) = \sum_{i,j\geq0}b_{ij}z_1^iz_2^j$ where
\beq\label{coeff}
b_{ij}= \frac{\alpha+\beta+\theta+j+2}{\alpha+\beta+\theta+2}\cdot\frac{(\alpha+\theta+2)_j}{(\theta+1)_j}\cdot\frac{(\alpha+\beta+\theta+j+3)_i}{i!}.
\eeq
Let $I_{m,n}:=\{(i,j)\in\mathbb Z\times\mathbb Z: i,j\geq 0, i\geq m\mbox{~or~}j\geq n\}$. We note that
$$K_{m,n}(z,w)= \d_{(i,j)\in I_{m,n}}b_{ij}z_1^iz_2^j\bar w_1^i\bar w_2^j.$$
One easily see that the set $\{z_1^m, z_2^n\}$ forms a minimal set of generators for the sheaf corresponding to $[\mathcal I_{m,n}]_{{\mathcal A}^2_{\alpha,\beta,\theta}(\mathbb B^2)}$. The reproducing kernel then can be decomposed as 
$$
K_{m,n}(z,w) = \bar w_1^mK_1^{m,n}(z,w)+\bar w_2^nK_2^{m,n}(z,w) \mbox{~for~ some~} r>0\mbox{~and~} w\in\Delta(0;r).
$$ 
Successive differentiation, using Leibnitz rule, gives
\beqa
K_1^{m,n}(z,w)|_{w=0} &=& \frac{1}{m!}\bar\partial_1^mK_{m,n}(\cdot,w)\}|_{w=(0,0)} = b_{m0}z_1^m\mbox{~and~}\\ K_2^{m,n}(z,w)|_{w=0} &=& \frac{1}{n!}\bar\partial_2^nK_{m,n}(\cdot,w)\}|_{w=(0,0)} = b_{0n}z_2^n.
\eeqa
Therefore
\begin{eqnarray*}
s_1(\theta_1) &=& b_{m0}z_1^m+\theta_1 b_{0n}z_2^n,
\end{eqnarray*}
where $\theta_1$ denotes co-ordinate for the corresponding open chart in $\mathbb P^1$. Thus
$$
{\parallel s_1(\theta_1)\parallel}^2~=~b_{m0}^2 \parallel z_1^m\parallel^2+b_{0n}^2 \parallel z_2^n\parallel^2 |\theta_1|^2 =~b_{m0}+b_{0n} |\theta_1|^{2}.
$$ 
Let
$a_{m,n}=b_{0n}/b_{m0}$. Let $\mathcal K_{m,n}$ denote the curvature corresponding
to the bundle $\mathcal L_{0,m,n}$ which is determined on the projective space $\mathbb P^1$ by the module $[\mathcal I_{m,n}]_{{\mathcal A}^2_{\alpha,\beta,\theta}(\mathbb B^2)}$. Thus we have
\begin{eqnarray*} 
{\mathcal K}_{m,n}(\theta_1) &=&\partial_{\theta_1}\partial_{\bar\theta_1} {\rm
ln}{\parallel s_1(\theta_1)\parallel}^2 ~=~
\partial_{\theta_1}\partial_{\bar\theta_1} {\rm
ln}(1+a_{m,n}|\theta_1|^2)\\&=&\partial_{\theta_1}\frac{a_{m,n}\theta_1}{1+a_{m,n}|\theta_1|^2}
= \frac{a_{m,n}}{(1+a_{m,n}|\theta_1|^2)^2}.
\end{eqnarray*} 
Let $\mathcal K_{m,n}'$ denote the curvature corresponding
to the bundle $\mathcal L_{0,m,n}'$ which is determined on the projective space $\mathbb P^1$ by the module $[\mathcal I_{m,n}]_{{\mathcal A}^2_{\alpha',\beta',\theta'}(\mathbb B^2)}$. As above we have
$$
{\mathcal K}_{m,n}'(\theta_1)~=~\frac{a_{m,n}'}{(1+a_{m,n}'|\theta_1|^2)^2}.
$$
Since the submodules $[\mathcal I_{m,n}]_{{\mathcal A}^2_{\alpha,\beta,\theta}(\mathbb B^2)}$  and $[\mathcal I_{m,n}]_ {{\mathcal A}^2_{\alpha',\beta',\theta'}(\mathbb B^2)}$  are unitarily equivalent, from Theorem \ref{proj}, it follows that $\mathcal K_{m,n}(\theta_1)={\mathcal K}_{m,n}'(\theta_1)$ for  $\theta_1$ in an open chart $\mathbb P^1$ and $m,n\geq N$. Thus
$$
\frac{a_{m,n}}{(1+a_{m,n}|\theta_1|^2)^2}~=~\frac{a_{m,n}'}{(1+a_{m,n}'|\theta_1|^2)^2}.
$$ This shows that $(a_{m,n} - a_{m,n}')(1+a_{m,n}a_{m,n}'|\theta_1|^2) = 0$. So $a_{m,n} = a_{m,n}'$ and hence 
\beq\label{p2}
\frac{b_{0n}}{b_{m0}} = \frac{b_{0n}'}{b_{m0}'}
\eeq
for all $m,n\geq N$. This also follows from the equation \eqref{sec}. It is enough to consider the cases $(m,n) = (N,N), (N,N+1), (N,N+2)$ and $ (N+1, N)$ to prove the Proposition.  From equation \eqref{p2}, we have 
\begin{eqnarray}\label{p3}
\frac{b_{(N+1)0}}{b_{N0}}=\frac{b_{(N+1)0}'}{b_{N0}'},~ \frac{b_{0(N+1)}}{b_{0N}}=\frac{b_{0(N+1)}'}{b_{0N}'} \mbox{~and~} \frac{b_{0(N+2)}}{b_{0(N+1)}}=\frac{b_{0(N+2)}'}{b_{0(N+1)}'}.
\end{eqnarray} Let $A=\alpha+\beta+\theta, B=\alpha+\theta$ and $C=\theta$. From equation \eqref{coeff}, we have 
$$
\frac{b_{(N+1)0}}{b_{N0}} = \frac{A+N+3}{N+1},~   \frac{b_{0(N+1)}}{b_{0N}}~=~\frac{A+N+3}{A+N+2}\cdot\frac{B+N+2}{C+N+1}
$$ 
and
$$
\frac{b_{0(N+2)}}{b_{0(N+1)}}~=~\frac{A+N+4}{A+N+3}\cdot\frac{B+N+3}{C+N+2}.
$$ 
From \eqref{p3}, it follows that $A=A'$ and 
\beq\label{p4}
BC'+B(N+1)+ C'(N+2) ~=~B'C+B'(N+1)+ C(N+2) ,
\eeq
\beq\label{p5}
BC'+B(N+2)+ C'(N+3) ~=~B'C+B'(N+2)+ C(N+3) .
\eeq
Subtracting
\eqref{p5} from \eqref{p4}, we get
$B - C= B'- C'$ and thus $\theta = \theta'$. Therefore $\frac{b_{0(N+1)}}{b_{0N}}=\frac{b_{0(N+1)}'}{b_{0N}'}$ implying  $B=B'$ and hence $\alpha = \alpha'$.
Lastly $A=A'$ and in consequence $\beta =\beta'$.\end{proof}

\subsection*{\sf Acknowledgement} We thank R. G. Douglas and M. Putinar for several very useful suggestions and many hours of fruitful discussion relating to this work.

\end{document}